\documentclass[11pt]{article}
\usepackage{amssymb,amsfonts,amsmath,amsthm}
\usepackage{color}
\definecolor{mh}{rgb}{0,0,0.75}

\usepackage{graphicx}

\usepackage[font=small,format=plain,labelfont=bf,up]{caption}
\newtheorem{prop}{Proposition}[section]

\newcommand{\nwc}{\newcommand}
\nwc{\R}{\mathbb R}
\nwc{\Z}{\mathbb Z}
\nwc{\N}{\mathbb N}
\newcommand{\ignore}[1]{}
\nwc{\eps}{\varepsilon}
\begin{document}
%
%
\title{ Instabilities and oscillations in coagulation equations with kernels of homogeneity one}
\author{ %
Michael Herrmann\thanks{ %
Westf\"alische Wilhelms-Universit\"at M{\"u}nster,
Institut f\"ur Numerische und Angewandte Mathematik
} %
\and %
Barbara Niethammer\thanks{ %
Rheinische Friedrich-Wilhelms-Universit\"at Bonn, Institut f\"ur Angewandte Mathematik
} %
\and{ %
Juan~J.L.~Vel\'{a}zquez}\footnotemark[2] %
} %
\maketitle
%
%
\begingroup
\let\thefootnote\relax\footnotetext{{\tt michael.herrmann@uni-muenster.de}, {\tt niethammer@iam.uni-bonn.de},
 {\tt velazquez@iam.uni-bonn.de}} 
\endgroup
%
%
%
%
\begin{abstract}
We discuss the long-time behaviour of solutions to Smoluchowski's coagulation equation with kernels  of homogeneity one, combining formal asymptotics, heuristic
arguments based on linearization, and numerical simulations.  
The case of what we call diagonally dominant  kernels is particularly interesting. Here 
one expects that the long-time behaviour is, after a suitable change of variables, the same as for the Burgers equation. 
However, 
for kernels that are close to the diagonal one we obtain instability of both, constant solutions  and traveling waves and in general no convergence to $N$-waves for 
integrable data. On the other hand,  for kernels not close to the diagonal one
 these structures are stable, but the traveling waves have strong oscillations. This has  implications on the approach towards an  
  $N$-wave for integrable data, which is also characterized by strong oscillations near the shock front.
 
\end{abstract}
%
%
%
 \quad\newline\noindent%
 \begin{minipage}[t]{0.15\textwidth}%
   Keywords: 
 \end{minipage}%
 \begin{minipage}[t]{0.8\textwidth}%
Smoluchowski's coagulation equation,  kernels with homogeneity one
 \end{minipage}%
 \medskip
 \newline\noindent

%
%
%
%
%
%
%
\section{Introduction}\label{S.introduction}
%
%
\paragraph{Smoluchowski's coagulation equation.}
In 1916 Smoluchowski derived a mean-field equation to describe coagulation in homogeneous gold solutions, which is nowadays used in a large variety of mass aggregation 
phenomena \cite{Smolu16}. It applies to a homogeneous dilute system of clusters that can coagulate by binary collisions to form larger clusters. 
If  $f(t,\xi)$ denotes the number density of clusters of size $\xi>0$ at time  $t$, then $f$ satisfies
\begin{equation}\label{eq1}
\partial_t f(t,\xi) = \tfrac 1 2 \int_0^{\xi} K(\xi{-}\eta,\eta) f(t,\xi{-}\eta) f(t,\eta)\,d\eta - f(t,\xi) \int_0^{\infty} K(\xi,\eta) f(t,\eta)\,d\eta\,,
\end{equation}
where the  so-called rate kernel $K$ is a nonnegative and symmetric function  which  describes the microscopic details of the coagulation process.
\par
If clusters are spherical, diffuse
by Brownian motion and coagulate quickly when they get within a certain interaction range, Smoluchowski \cite{Smolu16} derived the kernel 
\begin{align*}
K(\xi,\eta) = \big(\xi^{1/3} + \eta^{1/3}\big)\big(\xi^{-1/3} + \eta^{-1/3}\big).
\end{align*}
Further examples of kernels for a diverse range of
applications, such as aerosol physics, polymerization, growth of nanostructures, or astronomy, can be found in the survey articles  
 \cite{Drake72,Aldous99,Friedlander00}.
\par
The well-posedness of the initial value problem corresponding to \eqref{eq1} is by now quite well understood. It is also known that  if the kernel $K$ grows too fast at infinity, e.g. if $K$ is homogeneous of degree larger than one,
 then solutions to \eqref{eq1} exhibit the phenomenon of gelation, that is the loss of mass at finite time,
 which  is linked to the formation of infinitely large clusters. On the other hand,  if $K$ has homogeneity smaller than one or if  $K(\xi,\eta) \leq C(1+\xi +\eta)$  and if the initial
 data have finite mass, then  solutions to \eqref{eq1}  conserve the mass for all times \cite{LauMisch02}.
\par

A scale invariance of the equation leads to the so-called scaling hypothesis which suggests that the long-time behaviour of solutions
to \eqref{eq1} is universal and asymptotically described by self-similar solutions.
 This issue  is so far understood \cite{MePe04}  for the two  solvable kernels of homogeneity $\gamma \leq 1$,  the constant
 one and the additive one, $K(\xi,\eta)=\xi+\eta$. For these kernels  equation \eqref{eq1} can be  solved explicitly by   Laplace transform. 
For non-solvable kernels with homogeneity strictly  smaller than one, only existence results for self-similar
solutions  are available \cite{FouLau05,EMR05,NV12a,NTV15}, while questions about uniqueness and convergence to these self-similar solutions
have so far only been answered for a  few special cases \cite{LauNiVel16,NTV15}.

\paragraph{Kernels with homogeneity one.}
Our goal in this article is to investigate the long-time behaviour of solutions to \eqref{eq1} for kernels with homogeneity equal to one. One example is
$K(\xi,\eta)= \big(\xi^{1/3}+\eta^{1/3}\big)^3$ that has been derived for particles moving in a shear flow \cite{Smolu16}, others appear in gravitational coalescence or charged
aerosols \cite{Drake72}.
Such kernels represent the borderline case that separates gelation from self-similar coarsening and we expect additional
phenomena and technical difficulties. Apart from the solvable additive kernel, for which  complete results are available \cite{Bertoin02,MePe04},
no other kernel of homogeneity one has, at least to our knowledge, been studied in the mathematical literature.
Some useful insight into the properties of solutions based on formal considerations has been gained
in \cite{vanDoErnst88,Leyvraz03}, but not all aspects have been investigated there.
It is the goal of this article to provide more information on  what to expect on the long-time behaviour of solutions to the coagulation equation with kernels of homogeneity
one. We will give some rigorous results for a special case, the diagonal kernel, and provide several conjectures for the general case that we support by numerical simulations.

\paragraph{Self-similar solutions.}
For the  following considerations  we rewrite \eqref{eq1} in conservative form, that is as 
 \begin{equation}\label{eq2}
  \partial_t \big( \xi f(t,\xi)\big) = - \partial_{\xi} \Big( \int_0^{\xi} \,d\eta \int_{\xi-\eta}^{\infty} \,d\zeta K(\eta,\zeta) \eta f(t,\eta) f(t,\zeta)\Big)\,.
 \end{equation}

 We are interested in finding self-similar solutions of \eqref{eq2} and hence make the ansatz
 \begin{equation}\label{ss1}
  f(t,\xi)= \frac{1}{s(t)^2} \Phi(x)\,, \qquad x= \frac{\xi}{s(t)}\,.
 \end{equation}
If we plug \eqref{ss1} into \eqref{eq2} we obtain, using that the kernel has homogeneity one, the formulas 
\begin{equation}\label{sform}
 s(t)=e^{bt}\,, \qquad b>0\,,
\end{equation}
and
\begin{equation}\label{ssequation}
b\big(2x\Phi(x)+x^2 \Phi'(x)\big) = \partial_x \Big(\int_0^x \int_{x-y}^{\infty} K(y,z)y \Phi(y)\Phi(z)\,dz\,dy\Big)\,,
\end{equation}
which, after  integration with respect to $x$, gives
\begin{equation}\label{ssequation2}
bx^2\Phi(x) =\int_0^x \int_{x-y}^{\infty} K(y,z)y \Phi(y)\Phi(z)\,dz\,dy\,.
\end{equation}

\paragraph{A useful change of variables.}
It turns out  that the  following change of variables is useful. 
 We define
 \begin{equation}\label{newvariables}
  \xi=e^{X} \qquad \mbox{ and } \qquad u(t,X)= \xi^2 f(t,\xi)
 \end{equation}
 such that  \eqref{eq2} becomes 
 \begin{equation}\label{uequation}
  \partial_t u = - \partial_X \Big( \int_{-\infty}^X \int_{X+ \ln (1-e^{Y-X})}^{\infty}  K(e^{Y-Z},1) u(t,Y) u(t,Z)\,dZ\,dY\Big)\,.
  \end{equation}
In these new variables, scale invariant solutions correspond to traveling wave solutions. 
More precisely, if we make the ansatz $u(t,X)=G(X-bt)$, then $G$ must satisfy
\begin{equation}\label{Gequation}
\begin{split}
 bG(X)&= \int_{-\infty}^X \int_{X+ \ln (1-e^{Y-X})}^{\infty}  K(e^{Y-Z},1) G(Y) G(Z)\,dZ\,dY\\
 &=\int_{-\infty}^0 \int_{\ln (1-e^{Y})}^{\infty}  K(e^{Y-Z},1) G(Y+X) G(Z+X)\,dZ\,dY.
\end{split}
\end{equation}
Notice that the translation invariance that we used in the last step in \eqref{Gequation} is a  consequence of the fact that $K$ has homogeneity one.
The relation to the original variables is $\xi^2 f(t,\xi)=G(\ln \xi -bt)$ and $x^2\Phi(x)= G(X)$, respectively.
Furthermore note that the quantity that is preserved by the evolution \eqref{eq1},  the first moment of $f$,  now turns into  the integral of $u$, that is
$\int_0^{\infty} \xi f(t,\xi)\,d\xi = \int_{-\infty}^{\infty} u(t,X)\,dX$.

\par
In the following we will only consider traveling waves that can be related to self-similar
solutions of \eqref{eq1} which decay sufficiently fast.
Therefore, we will restrict ourselves to solutions of \eqref{Gequation} which satisfy  $G(\infty ) =0$.
The precise role of the parameter $b$ will be discussed below, as it depends on the type of kernel as well as on the type of solutions that we will consider.

\paragraph{Class-II kernels.}
To proceed, we need to distinguish between two types of kernels, a fact that has
already been noticed in \cite{vanDoErnst88,Leyvraz03}.
In the first case, called  class-II kernels  in   \cite{vanDoErnst88},
one has $K(\xi,1) \to k_0>0$ as $\xi \to 0$.
In this case  we notice that the integral 
\[\int_{-\infty}^X \int_{X+ \ln (1-e^{Y-X})}^{\infty}  K(e^{Y-Z},1)\,dZ\,dY
 \]
 is not finite. Consequently, if a solution
 $G$ to \eqref{Gequation} exists for some $b>0$, then it must at least satisfy $G(X) \to 0$ as $|X| \to \infty$.

The most prominent example of a class-II kernel is the additive kernel, 
for which  it is known that there exists a whole family of self-similar solutions with
finite mass \cite{MePe04}. One of them has exponential decay, the others decay algebraically  such that the second moment is infinite. 
This family of solutions can be parameterized by the parameter $b$ in \eqref{Gequation}. In fact, if one normalizes the integral of $G$ to one there
is a one-to-one correspondence between $b$ and the decay behaviour of the solutions to \eqref{Gequation}. The result in \cite{MePe04} provides the existence
of self-similar solutions for any $b \geq 2$.

The question whether an analogous result holds for other class-II kernels is presently open. In Section \ref{S.classII} we give
self-consistent arguments to describe the expected decay behaviour of self-similar solutions, see the left panel in Figure \ref{fig:cartoon}, and formulate a conjecture that self-similar solutions 
exist for $b$ larger than a critical number $b_*$ that depends on the kernel.

\begin{figure}[h!]
  \centering
  \includegraphics[width=.4\textwidth]{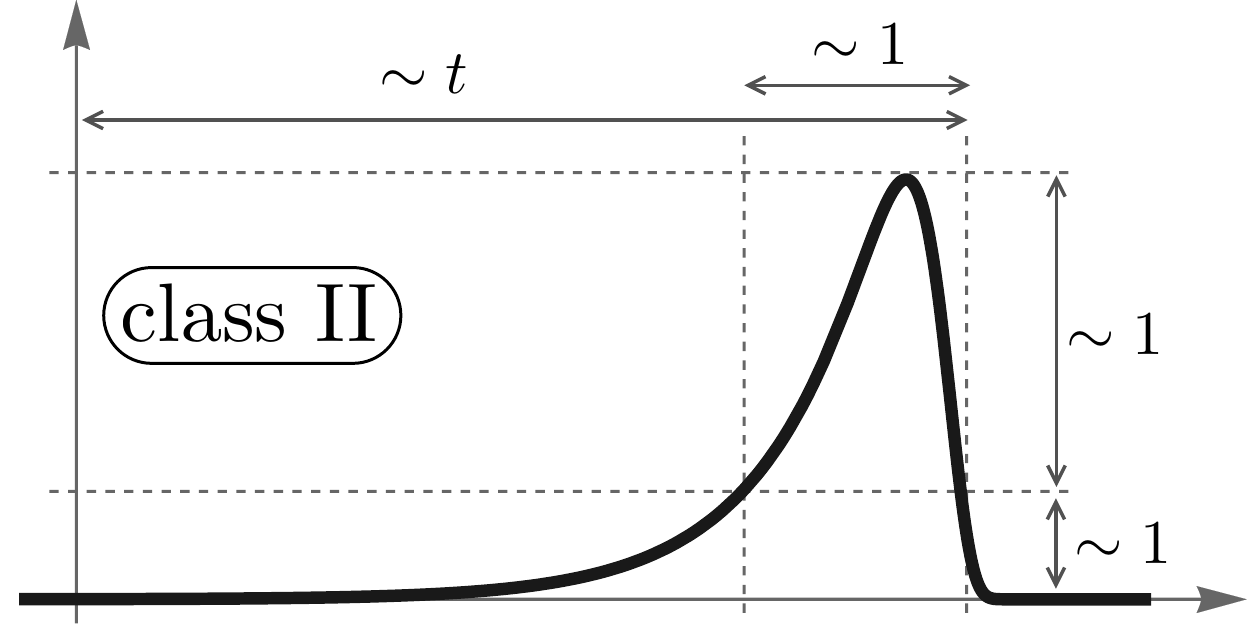}
  \hskip1cm
  \includegraphics[width=.4\textwidth]{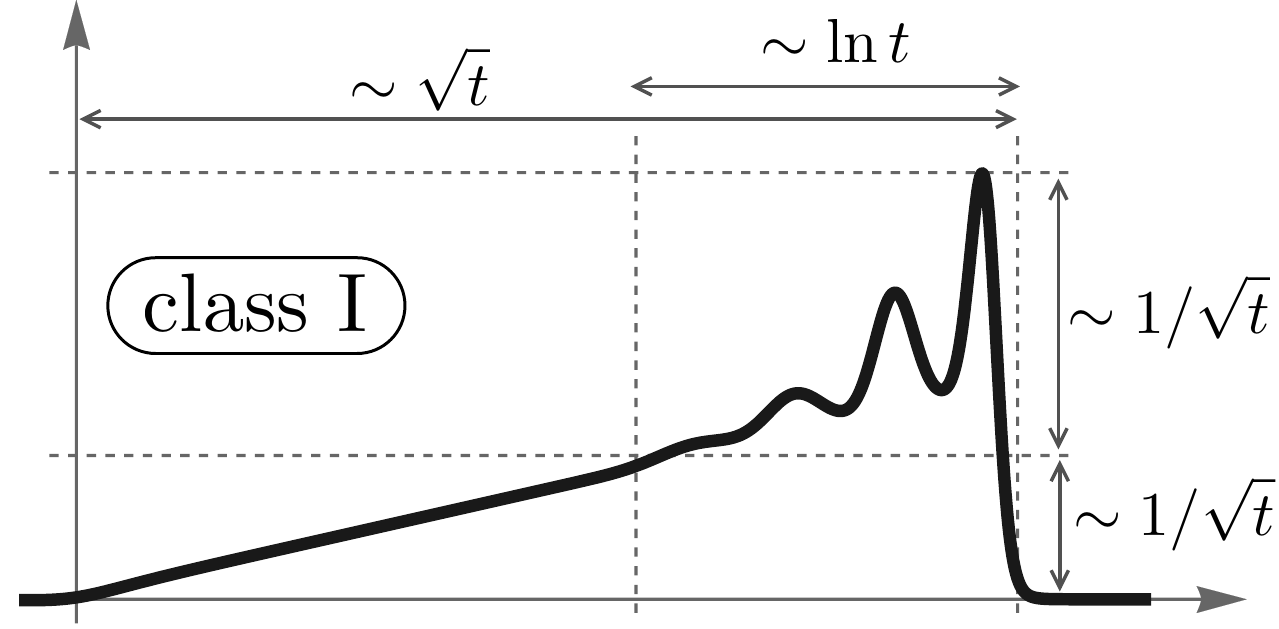}
\caption{Cartoons of the different long time behavior for data with finite mass.
\emph{Left panel}: For class-II kernel we expect to find self-similar solutions, which are transformed into traveling waves by the rescaling \eqref{newvariables}.
\emph{Right panel}: Traveling waves exist also for class-I kernels but carry infinite mass and can exhibit strong oscillations. The long-time 
behaviour is -- at least for some kernels  -- governed by an N-wave solution with attached traveling wave at the front.
}
  \label{fig:cartoon}
\end{figure}

\paragraph{Class-I kernels.}
We now consider kernels that satisfy $\lim_{x \to 0} K(x,1)=0$ and make the additional assumption, as has been done in \cite{vanDoErnst88},
that  the kernel is asymptotically a power law for small $x$, that is
\begin{equation}\label{assumption}
 K(x,1) \sim c_{\alpha} x^{\alpha} \qquad \mbox{ as } x \to 0 \quad \mbox{ for some } \alpha>0 \mbox{ and } c_{\alpha} \geq 0\,.
\end{equation}
Such kernels are called class-I kernels in the notation of \cite{vanDoErnst88},  we sometimes also call 
such  kernels   {\it diagonally dominant}. 
 A limiting case is the diagonal kernel 
\begin{align}
\label{DiagonalKernel}
K(x,y)=x^2 \delta_{x-y}
\end{align}
 for which 
only particles of the same size are allowed to coagulate. Note that this kernel has homogeneity one since the Dirac distribution has homogeneity minus one.
 Further examples are  given by the following  family of kernels (see Figure \ref{fig:kernels})
 \begin{equation}\label{kernelfamily}
  K_{\alpha}(x,y) = c_{\alpha} x^{\alpha}y^{\alpha}\big(x+y)^{1-2\alpha}\,, \qquad \alpha >0\,,
 \end{equation}
which interpolates between the additive kernel ($\alpha=0$) and the diagonal kernel ($\alpha \to \infty$). Here the constant $c_{\alpha}$ is a suitable normalization constant that
 will be chosen  later in Section \ref{Ss.general}.

\begin{figure}[h!]
  \centering
  \includegraphics[width=.4\textwidth]{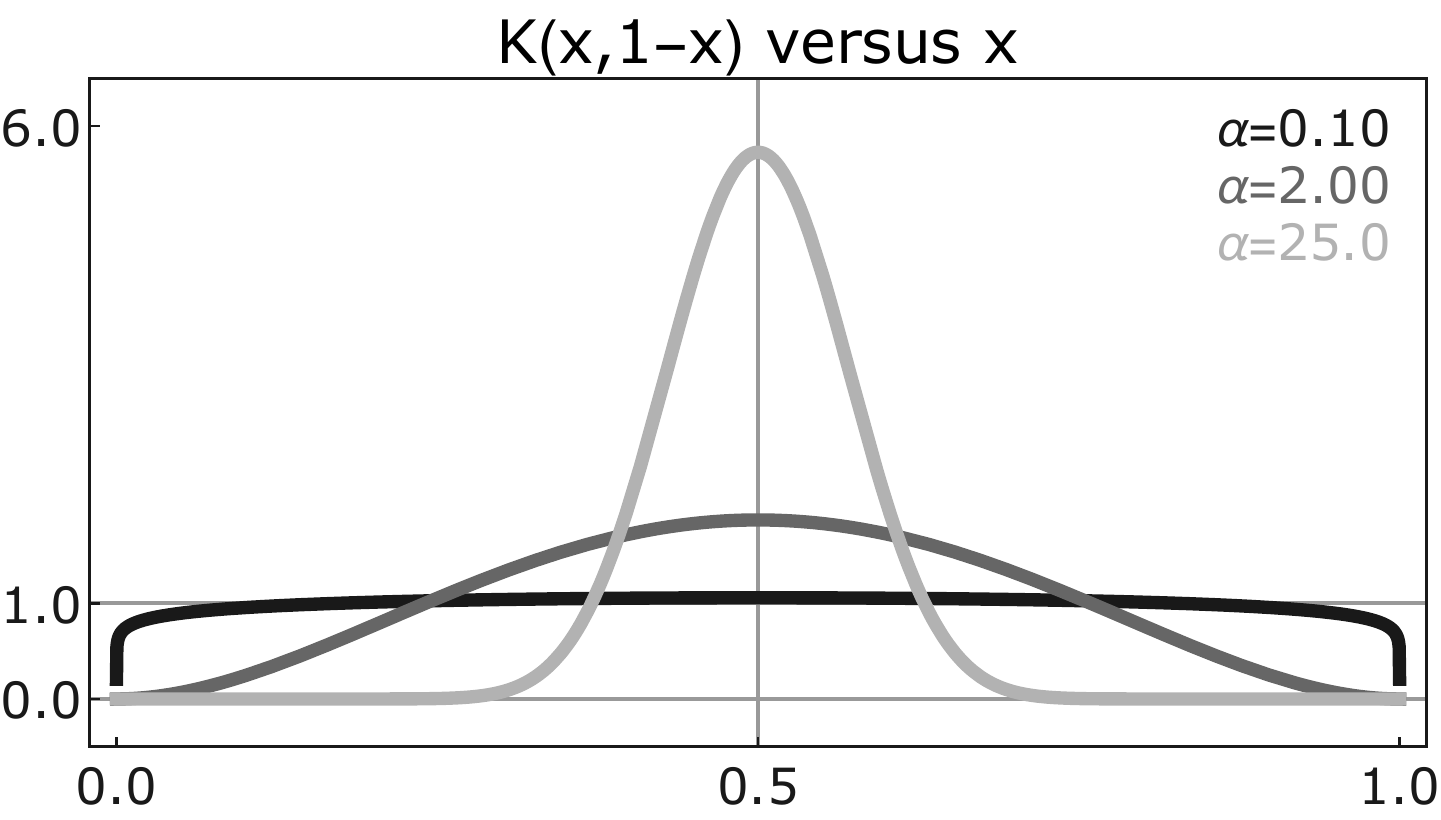}
  \caption{Three  kernels  of the family in \eqref{kernelfamily} with $c_{\alpha}=\frac{\Gamma(2+2\alpha)}{\Gamma(1+\alpha)^2}$ 
  such that $\int_0^1 K(x,1{-}x)\mathrm{d}x=1$. }
  \label{fig:kernels}
\end{figure}

 In contrast to kernels of class II,  for kernels of class I  satisfying \eqref{assumption}, we have
\begin{equation}\label{finiteintegral}A:=\int_{-\infty}^0 \int_{\ln (1-e^{Y})}^{\infty}  K(e^{Y-Z},1)\,dZ\,dY< \infty.
 \end{equation}
 Hence, the only self-consistent behaviour of a solution $G(X)$ as $X \to -\infty$ is 
$\lim_{X \to -\infty} G(X)=G(-\infty)=b A^{-1}>0$.
However, then the integral over $G$ is not finite, and this contradicts the assumption of finite mass, a fact that has  already been noticed
in \cite{vanDoErnst88,Leyvraz03}.
As a consequence, the ansatz \eqref{ss1} is inconsistent. Nevertheless,  solutions of \eqref{Gequation} are also of interest since they  correspond
to traveling wave solutions in the variable $\ln \xi$ for
$\xi^2 f(t,\xi)$, but they cannot appear as the large time limit for solutions with finite mass. 
Notice also that here the parameter $b$ determines the asymptotic jump
of the traveling wave and can without loss of generality be set equal to one.

In Section \ref{S.classI} we  first provide an argument based on formal asymptotics that the long-time behaviour of solutions 
with finite mass is to leading order the same as the long-time behaviour
of solutions to  the inviscid Burgers equation. As a first step we consider the diagonal kernel in Section \ref{Ss.diagonal},  for which we can show
rigorously that solutions converge to an $N$-wave in the long-time limit. Next, we consider the family of kernels $K_{\alpha}$ in \eqref{kernelfamily} for different values of $\alpha$.
In Sections \ref{Ss.stabilityneardiagonal} and \ref{Ss.stabilityalpha} we discuss the stability resp. instability of constant solutions before we turn to traveling waves
in Section \ref{Ss.traveling}. Numerical simulations suggest that  the wave profiles are monotone for large $\alpha$, i.e. for kernels close to the diagonal one,
but oscillatory for small $\alpha$, a fact
that can be explained by a linearization argument. Finally, in Section \ref{Ss.nwaves} we investigate, also numerically, the long-time behaviour
of solutions with integrable data.
It turns out that, at least for moderate $\alpha$, the solution converges to an $N$-wave,
but the transition at the shock front is given by a traveling wave and is hence oscillatory for some range 
of kernels. An illustration is given in the right panel of Figure \ref{fig:cartoon}.

 \section{Class II kernels}\label{S.classII}

Before we consider general class-II kernels, we briefly recall the results on self-similar solutions for the additive kernel, formulated in the variables
\eqref{newvariables}.
%
\subsection{The additive kernel}
%
 
 \begin{prop}[Section 6.1 in \cite{MePe04}]
Suppose that $K(\xi,\eta)=\xi+\eta$ and that $b=\frac{1{+}\rho}{\rho}$. Then there exists for  any $\rho \in (0,1]$  a solution  $G_{\rho}(X)$ to \eqref{Gequation}
with  unit mass, that is 
 $\int_{-\infty}^{\infty} G_{\rho}(X)\,dX=1$.
For $\rho=1$ the solution is 
\begin{equation}\label{G1}
 G_1(X)= \tfrac{1}{\sqrt{2\pi}} e^{\frac{X}{2}} e^{-\frac{e^{X}}{2}}\,,
\end{equation}
while for $\rho \in (0,1)$ the solution is given by 
\begin{equation}\label{Grhorep}
G_{\rho}(X)= \frac{1}{\pi} \sum_{k=1}^{\infty} \frac{(-1)^{k-1}}{k!} e^{k\frac{\rho}{\rho+1}X }\Gamma\Big(1+k-k\frac{\rho}{1+\rho}\Big)\sin\Big(k \pi \frac{\rho}{1+\rho}\big)\,.
\end{equation}
The asymptotics of $G$ for $|X|\to \infty$ are
\begin{equation}\label{Grhominus}
 G_{\rho}(X) \sim  \frac{\sin\big( \frac{\pi \rho}{1+\rho}\big) \Gamma \big( \frac{1}{1+\rho}\big)}{\pi (1+\rho)} e^{\frac{\rho}{1+\rho}X} \qquad \mbox{ as } X \to -\infty
\end{equation}
and
\begin{equation}\label{Grhoplus}
 G_{\rho}(X) \sim  \frac{\Gamma(1+\rho)\sin(\pi(1-\rho)}{\pi}  e^{-\rho X} \qquad \mbox{ as } X \to \infty\,.
\end{equation}
\end{prop}
Thus, for any $b \in [2,\infty)$ a self-similar solution with unit mass exists. 
Notice, that the rescaled function $G_{m,\rho}(X):=mG_{\rho}(X)$ has mass $m$ and solves \eqref{Gequation} with $b=m \frac{1{+}\rho}{\rho}$. 

For the additive kernel it is also easily seen that whenever $G$ satisfies $\int_{-\infty}^{\infty} e^{X} G(X)\,dX < \infty$, then it must hold that $\rho=1$. This follows, since multiplying
\eqref{Gequation} by $e^{X}$ and integrating gives $\frac{\rho+1}{\rho} M_2= 2M_1 M_2=2M_2$, where we use the notation $M_i=\int_{-\infty} e^{(i-1)X}G(X)\,dX$. (These quantities correspond to the
$i$-th moments of $\Phi(x)$).
%
\subsection{Nonsolvable class II kernels}
%
 
We now assume that $K$ is a general kernel of homogeneity one that satisfies $\lim_{\xi \to 0} K(\xi,1) =k_0 >0$. Without loss of generality we assume in the following $k_0=1$.

\paragraph{Heuristics of asymptotic behaviour.}

 We first show that, if a solution $G_{\rho}$ to \eqref{Gequation} exists that satisfies $G_{\rho}(X)\sim e^{- \rho X}$ as $X \to \infty$,  then necessarily $b= \frac{\rho+1}{\rho}$.
 (Notice that in general we obtain the relation $b=k_0\frac{1+\rho}{\rho}$).
For that purpose we split the integral 
 \[\int_{-\infty}^X \int_{X+\ln(1-e^{Y-X})}^{\infty}\,dZ\,dY = \int_{-\infty}^X \int_{X}^{\infty}\,dZ\,dY + \int_{-\infty}^X \int_{X+\ln(1-e^{Y-X})}^{X}\,dZ\,dY.
  \]
Since $K(\xi,1) \approx 1$ for small $\xi$ and if $G_{\rho}(X) \sim e^{-\rho X}$ as $X \to \infty$, then (recall that $\int_{-\infty}^{\infty} G(X)\,dX=1$)
 \[
  \int_{-\infty}^X \int_{X}^{\infty} K(e^{Y-Z},1) G_{\rho}(Y) G_{\rho}(Z) \approx  \int_{-\infty}^X G_{\rho}(Y)\,dY \int_{X}^{\infty} G_{\rho}(Z)\,dZ
  \approx \frac{1}{\rho} e^{-\rho X}
 \]
as $X \to \infty$.
Furthermore, changing the order of integration, 
\[
\begin{split}
 \int_{-\infty}^X \int_{X+\ln(1-e^{Y-X})}^{X}&K(e^{Y-Z},1) G_{\rho}(Y)G_{\rho}(Z)\,dZ\,dY\\
 &= \int_{-\infty}^X G_{\rho}(Z)\int_{X+\ln(1-e^{Z-X})}^{X}K(e^{Y-Z},1)G_{\rho}(Y)\,dY\,dZ\\
 & \approx G_{\rho}(X) \int_{-\infty}^X G_{\rho}(Z) \int_{X+\ln(1-e^{Z-X})}^{X}K(e^{Y-Z},1)\,dY\,dZ \\
& \approx  G_{\rho}(X) = e^{-\rho X} \qquad \mbox{ as } X \to \infty.
\end{split}
\]
Here we used that
\[
 \int_{X+\ln(1-e^{Z-X})}^{X}K\big(e^{Y-Z},1\big)\,dY = \int_{x-z}^x\frac{1}{y} K\Big( \frac{y}{z},1\Big) \,dy = \frac{1}{z} \int_{x-z}^x K\Big(1,\frac{z}{y}\Big)\,dy \approx
 \frac{1}{z} \int_{x-z}^x\,dy=1\,.
\]
As a consequence, $b= \frac{1+\rho}{\rho}$ if such a solution exists.

Now assume that $G_{\rho}(X) \sim e^{a X}$ as $X \to -\infty$. Then, as above, 
 \[
  \int_{-\infty}^X \int_{X}^{\infty} K(e^{Y-Z},1) G_{\rho}(Y) G_{\rho}(Z) \approx \int_{-\infty}^X G_{\rho}(Y)\,dY \int_{X}^{\infty} G_{\rho}(Z)\,dZ
  \approx \frac{1}{z} e^{-a X} \sim \frac{1}{a}e^{a X} 
 \]
as $X \to -\infty$. The term  $\int_{-\infty}^X \int_{X+\ln(1-e^{Y-X})}^{X}K(e^{Y-Z},1) G_{\rho}(Y)G_{\rho}(Z)\,dZ\,dY$ give in this case  a contribution of higher order.
Hence, we have that $a = \frac{\rho}{1+\rho}$.

Thus, we see that if solutions for a given $b>0$ exist and if they have some exponential behaviour as $X \to \pm \infty$, then the relation between $b$ and the exponents is the 
same as in the case of the additive kernel.

 \paragraph{Conjecture on existence of solutions.}
 For a general class-II kernel $K$, we  conjecture that  there is a critical $\rho_*\in (0,\infty]$, depending on $K$,  such that for any $\rho \in (0,\rho_*)$ there
 exists a solution $G_{\rho}$ to \eqref{Gequation} that satisfies for $\rho \in (0,\rho_*)$ that $G_{\rho}(X) \sim e^{-\rho X}$ as $X \to \infty$, while for $\rho=\rho_*$ it decays double
 exponentially. If $\rho_*<\infty$ we furthermore conjecture, that there exists $\rho_{**} \geq \rho_*$ such that  there is no nonnegative solution to \eqref{Gequation} for $\rho>\rho_{**}$. 

 \medskip
 Notice that the 'fat-tail' solutions for $\rho<\rho_*$ are different from the self-similar solutions with fat tails for kernels with homogeneity $\gamma <1$. In the latter
 case, if $\gamma \geq 0$ and some structural assumptions on the kernel are satisfied, there exist
 self-similar solutions \cite{NV12a,NTV16} of  the form $f(t,\xi) = t^{r}F(\xi t^{-s})$ with $F(z) \sim z^{-(1+\rho)}$ with
 $\rho \in (\gamma,1)$, $s=\frac{1}{\rho-\gamma}$ and $r= 1+(1+\gamma) s$. Hence $f(t,\xi) \sim 
 A \xi^{-(1+\rho)}$ as $\xi \to \infty$ with a time-independent constant $A$. In contrast to that, consider 
 self-similar solutions for class II kernels with homogeneity one
 with a profile
 that decays as $x^{-(2+\rho)}$. Such solutions  have time dependent tails of the form $A(t)x^{-(2+\rho)}$ with $A(t)=e^{M_1(1+\rho)t}$, where $M_1$ denotes the mass.
 As a consequence, we expect that any existence proof of such solutions should be different from the ones
 that exist for 'stationary' fat tails.
 
 \medskip
 We conclude this section by a justification of our  conjecture  that the critical $\rho_*$ is in general not equal to one, but rather depends on the details of 
 the kernel.
 For that purpose we consider a kernel that is a perturbation of the additive one and look for solutions
 that are perturbations of the explicit solution of the additive kernel for  $b=2$, given in \eqref{G1}. We will see that in general this perturbation will have
 a polynomial decay if we fix $b=2$. In order to get exponential decay it is therefore necessary to change $b$ as well, which leads to a critical value $\rho_*$ that
 is different from one and possibly even infinity.
 
 For the corresponding computations it is more convenient to go back to the original self-similar variable $x$. This will allow us to obtain the properties of
 the solution to the linearized problem by using the Laplace transform.
We are looking for solutions of  \eqref{ssequation2} where $K=K_{\eps}$ is a   perturbation of the additive kernel,
 \begin{equation}\label{perturb1}
  K_{\eps}(x,y)=x+y + \eps W(x,y)\,,
 \end{equation}
where $W$ has homogeneity one. We assume that $W$ is smooth and has compact support.

We linearize around the explicit solution for $\rho=1$ for the additive kernel
\begin{equation}\label{perturb2}
 \bar \Phi(x) = \frac{1}{\sqrt{2\pi}} x^{-3/2} e^{-\frac{x}{2}}\,,
\end{equation}
that is we make the ansatz
\begin{equation}\label{perturb3}
 \Phi(x)=\bar \Phi(x) + \varphi(x) \qquad \mbox{ with } \qquad \int_0^{\infty} x \varphi(x)\,dx=0\,.
\end{equation}
If we plug \eqref{perturb3} into \eqref{ssequation2} with $b=2 + \mu {\eps}$ and neglect higher order terms, we find the linear equation for $\varphi$,
\begin{equation}\label{perturb4}
\begin{split}
 2 x^2 \varphi(x)& =\int_0^x \int_{x-y}^{\infty} y(z+y) \big( \bar \Phi(y) \varphi(z) + \bar \Phi(z) \varphi(y)\big)\,dz\,dy \\
 &\qquad \qquad + \eps \int_0^x\int_{x-y}^{\infty} y W(y,z)\bar \Phi(y) \bar \Phi(z)\,dz\,dy - \mu \eps x^2 \bar \Phi(x)\,,
 \end{split}
\end{equation}
where the  second line of \eqref{perturb4} contains the  source terms.

Since $\bar \Phi$ is not integrable, we expect that the same is true for  $\varphi$ and hence that 
its Laplace transform at zero is not defined. Instead, we introduce
\begin{equation}\label{perturb5}
 \tilde \varphi(p) = \int_0^{\infty} \big(1-e^{-pz}\big)\varphi(z)\,dz\,,
\end{equation}
multiply \eqref{perturb4} with $e^{-px}$ and integrate. We call $U(p)=\int_0^{\infty} \big(1-e^{-pz}\big)\bar \Phi(z)\,dz$ and notice, that the properties
of $\bar \Phi$ imply that $U(0)=0$ and $\frac{d}{dp} U(0)=1$. We obtain
\begin{equation}\label{perturb6}
 \Big( \frac{U(p)}{p}-2\Big) \frac{d^2}{dp^2} \tilde \varphi(p) = - \frac{1}{p}  \frac{d}{dp} \tilde \varphi(p) - 
 \frac{1}{p} \frac{d^2}{dp^2} U(p) \tilde \varphi + \tilde \lambda(p)+ \mu\eps  \frac{d^2}{dp^2} U(p)  \,, 
\end{equation}
where $\tilde \lambda$ contains the transforms of the source terms coming from $W$. 

As $p \to 0$ we obtain to leading order, since $U(p) \sim p$ as $p \to 0$, that
\begin{equation}\label{perturb7}
 \frac{d^2}{dp^2}\tilde \varphi - \frac{1}{p} \frac{d}{dp}\tilde \varphi + \frac{\kappa}{p} \tilde \varphi = \tilde \lambda (0)+\mu\eps \kappa\,,
\end{equation}
where $\kappa:= \frac{d^2}{dp^2}U(0)$. The solution of \eqref{perturb7}
 behaves as $\tilde \varphi \sim p^2 \ln p$ as $p \to 0$, which corresponds to $\varphi(x) \sim x^{-3}$ as $x \to \infty$. 
Hence, in order to obtain an exponentially decaying solution, one needs in general to choose $\mu$ different from zero in order to cancel this slow decay behaviour.
This indicates that the critical value $\rho_*$ is not necessarily equal to one as in the case of the additive kernel, and
is determined by an  eigenvalue problem, which is to find the critical $b$ in \eqref{Gequation} that yields fast decay of the solution.
It remains a challenging task to determine this critical $b$, at least for general kernels that are not perturbations of the solvable one.

%
\section{Class-I  kernels} \label{S.classI}
%
%
\subsection{Heuristics}\label{Ss.heuristics}
%
We have seen in the Introduction that for class-I kernels there are no self-similar solutions with finite mass.  In order to 
understand what happens to a solution with finite mass  in the long time limit, it is  convenient to look at the formulation \eqref{uequation}. 
 Together with the requirement that the integral of $u$ is finite, it suggests to consider the  rescaled function  $u_{\eps}(\tau,\tilde X) = 
 \frac{1}{\eps} u(\frac{\tau}{\eps^2},\frac{\tilde X}{\eps})$,  where $0<\eps\ll1$ is a scaling parameter. This yields
 \begin{equation}\label{uepsequation}
  \begin{split}
 \partial_{\tau} u_{\eps}& = - \partial_{\tilde X}\Big( \frac{1}{\eps^2} \int_{-\infty}^{\tilde X} \int_{\tilde X + \eps \ln \big( 1-e^{\frac{Y-\tilde X}{\eps}}\big)}^{\infty}
  K\big(e^{\frac{Y-Z}{\eps}},1\big) u_{\eps}(Y) u_{\eps}(Z)\,dZ\,dY\Big)\\
  &= - \partial_{\tilde X}
\Big( \frac{1}{\eps^2} \int_{-\infty}^{0} \int_{\eps \ln \big( 1-e^{\frac{Y}{\eps}}\big)}^{\infty}
 K\big(e^{\frac{Y-Z}{\eps}},1\big) u_{\eps}(\tilde X + Y) u_{\eps}(\tilde X+Z)\,dZ\,dY\Big)\\
  &\approx - c_0\partial_{\tilde X} \big(u_{\eps}(\tilde X)^2\big)\,
    \end{split}
 \end{equation}
 and we conclude that $u_{\eps}$ approximately solves the Burgers equation.
 In \eqref{uepsequation} we used in the last step that for any continuous function $\varphi$ with compact support we have
 \[
 \begin{split}
  \int_{-\infty}^{0} &\int_{\eps \ln \big( 1-e^{\frac{Y}{\eps}}\big)}^{\infty}
 \frac{ K\big(e^{\frac{Y-Z}{\eps}},1\big) }{\eps^2}\varphi(Y)\varphi(Z)\,dZ\,dY\\
 &=
 \int_{-\infty}^0 \int_{\ln(1-e^Y)}^{\infty} K\big(e^{y-z},1\big)\varphi(\eps y) \varphi(\eps z)\,dy\,dz \to c_0 \varphi^2(0)
 \end{split}
 \]
 as $\eps \to 0$ where $c_0=\int_{-\infty}^0 \int_{\ln(1-e^Y)}^{\infty} K\big(e^{y-z},1\big)\,dy\,dz<\infty $. In other words, on the domain of integration $\{Y<0,Z>\eps 
 \ln (1-e^{\frac{Y}{\eps}})\}$ the rescaled kernels
 $\eps^{-2}K \big(e^{\frac{Y-Z}{\eps}},1\big) $ converge to a multiple of a Dirac distribution.
 Of course, it is far from obvious that $u_{\eps}$ is sufficiently smooth such that the previous argument applies and that hence
 the long-time behaviour of finite mass solutions should be the same as for the Burgers equation. It is even less obvious that solutions to the coagulation equation
 provide an approximation to solutions of the Burgers equation that yield entropy solutions in the limit.

 Before we are going to discuss several aspects of solutions of the coagulation equation for $K_{\alpha}$ as in \eqref{kernelfamily} and present corresponding simulations in
 Section \ref{Ss.general}, we will first
  consider  in Section \ref{Ss.diagonal} the special case of the diagonal kernel
  with  homogeneity one. In this case  one can describe the long-time behaviour of solutions rigorously.

%
\subsection{The diagonal kernel}\label{Ss.diagonal}
%
 The diagonal kernel with homogeneity one is given by 
\eqref{DiagonalKernel}
for which   \eqref{eq1} reduces to 
$ \partial_t f(t,\xi)= \big(\xi/2\big)^2 f(t,\xi/2)^2 - \xi^2 f(t,\xi)^2$. In this case it is more convenient to introduce the new variables via 
$x=2^X$ and rescale time by $\ln2$, such that the equation for $u(t,X)=\xi^2f(t/\ln2,\xi)$
reads
\begin{equation}\label{diagonal1}
 \partial_t u(t,X) = u(t,X{-}1)^2 - u(t,X)^2\,.
\end{equation}
Hence the evolution in a point $X$ depends only on 
the evolution of the discrete values $X{-}j$ with $ j \in \mathbb{N}$. We  consider first the infinite system
\begin{equation}\label{diagonal2}
 \dot u_j(t)= u_{j-1}^2(t) - u_j^2(t)\,, \qquad j \in \mathbb{Z},
\end{equation}
which can also be interpreted as an upwind discretization of the Burgers equation. 
\paragraph{Formal asymptotics.}
Equation \eqref{diagonal2} has already been  analyzed by BenNaim and Krapvisky \cite{BenNaimKrap12} via formal asymptotics and 
numerical analysis. They consider two type of data: nonnegative integrable data with unit mass on the one hand and a decreasing step function, connecting the values one and zero, on the other
hand. In the second case, 
they predict convergence of the solution to a monotone traveling wave $G$  which  satisfies  $G'(z)=G(z{-}1)^2-G(z)^2$. $G$ has the property that 
 $G(z) \sim e^{bz}$ as $z \to -\infty$, while it decreases double exponentially, that is $G(z)\sim 2^z e^{-\gamma 2^z}$ as $z \to \infty$. Figure \ref{fig:twdiagonal} shows
convergence to the traveling wave for Riemann data.

\begin{figure}[h!]
  \centering
  \includegraphics[width=.95\textwidth]{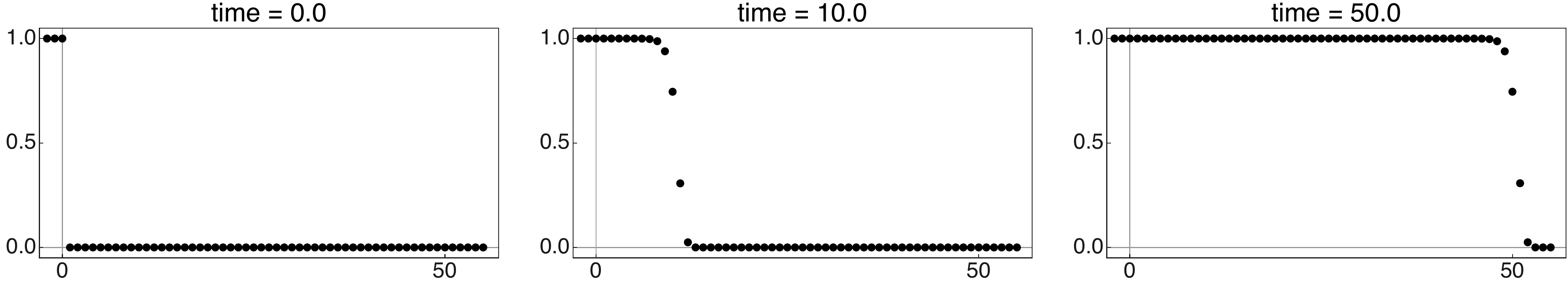}
  \caption{ Numerical solution to the Burgers lattice \eqref{diagonal2} with Riemann initial data: snapshots of $u_j(t)$ for several times $t$. For $t\to\infty$ 
  the lattice data approach a monotone travelling wave which propagates with speed $1$.}
  \label{fig:twdiagonal}
\end{figure}

 \medskip
 In the  case of integrable data, BenNaim and Krapivsky predict that for large times 
$u_j(t) \sim A_j t^{-1}$ for $0 \leq j \leq j_*=2\sqrt{t}-1/4\ln t$, where $A_j \sim \frac{j}{2} + \frac{\ln j}{4} + O(1)$, that is $u_j$ behaves to leading order as the $N$-wave with
the same mass as the initial data, but there are logarithmic corrections. Furthermore, the behavior at the front $j_*$ is predicted to be given by $u_j(t)\sim \frac{1}{\sqrt{t}} G(j-j_*)$, where $G$ is the
traveling wave profile studied before. 

We now give the outline of a short proof that establishes convergence to an $N$-wave of solutions to \eqref{diagonal2} for summable nonnegative
initial data.

\paragraph{Rigorous proof of convergence to an $N$-wave for the lattice \eqref{diagonal2}.}
We consider solutions to \eqref{diagonal2} with initial data $(u_j^0)$ that satisfy
\begin{equation}\label{data}
 u_j^0 \geq 0 \qquad \mbox{ for all } j \in \Z \qquad \mbox{ and } \qquad \sum_{j\in \Z} u_j^0 = M\,.
\end{equation}

We introduce the function $U_{\eps}$ as the piecewise constant function given by
\begin{equation}\label{Uepsdef}
 U_{\eps}(\tau,x)= \frac{1}{\eps} u_j(t)\,, \qquad \mbox{ with } \;\tau= \eps^2 t \quad \mbox{ and} \quad j=\Big \lfloor \frac{x}{\eps}\Big \rfloor\,.
\end{equation}

Furthermore, we denote by $N$ the continuous nonnegative $N$-wave with unit mass, 
that is
\begin{equation}\label{Ndef}
N(x;M)= \frac{x}{2} \chi_{[0,2\sqrt{ M}]}(x)\,.
\end{equation}

\begin{prop}\label{P.nwave} We have
\begin{equation}\label{mainresult}
\sup_{\tau \in (0,T)} \int_{\R}\Big| U_{\eps}(\tau,x)-\frac{1}{\sqrt{\tau}}N\big (\frac{x}{\sqrt{\tau}}\,;M\big)\Big|\,dx \to 0 \qquad \mbox{ as } \eps \to 0 
\end{equation}
or equivalently
\begin{equation}\label{mainresult2}
 \sum_j \Big| u_j(t)-\frac{1}{\sqrt{t}} N\Big(\frac{j}{\sqrt{t}};M\Big)\Big| \to 0 \qquad \mbox{ as } t \to \infty\,.
\end{equation}

\end{prop}
\begin{proof}
We follow the strategy employed, for instance, in \cite{Ignatetal15}, where the convergence of different discretization schemes of the Burgers equation is considered. 
Without loss of generality we assume now $M=1$. We also note that the maximum principle implies that $u_j(t) \geq 0$ for all $j \in \Z$ and $t >0$. Furthermore
we have $\sum_j u_j(t)=\int U_{\eps}(\tau,x)\,dx =1$. 

\smallskip
{\it Step 1: (Entropy condition)} We first show that
\begin{equation}\label{step1}
  \big( u_{j+1}(t)-u_j(t)\big)_+ \leq \frac{1}{\frac{1}{\sup_j (u_{j+1}^0-u_j^0)}+t}\,\qquad \text{for all} \quad j\in\Z\quad\text{and}\quad t\geq0.
\end{equation}
If we define $w_j:=u_{j+1}-u_j$ then $w_j$ satisfies
\[
  \dot w_j = -w_j^2 -w_{j-1}^2 -2u_j(w_j-w_{j-1})
\]
and the statement follows by comparison with the solution of the ODE $y'+y^2=0$.

\smallskip
{\it Step 2: (Decay estimate)}  We establish the temporal decay rate
\begin{equation}\label{step2}
 \sup_j u_j(t) \leq \frac{C}{\sqrt{t}}\,.
\end{equation}
The main idea is that \eqref{step1}  bounds  the increase of $u_j$ in $j$ by $\frac{1}{t}$. Hence, if  $\alpha:= \max_j u_j(t)$ and  $k$ is an index in which the maximum
is attained, then $u_j$ is bounded below by the line $\alpha +\frac{j-k}{t}$. Then we have with $k_0=\lfloor k-\alpha t\rfloor$ that
\[
 1 \geq \sum_{k_0}^k u_j \geq \sum_{k_0}^k \alpha + \frac{j-k}{t} \geq \alpha^2t - \sum_{k_0}^k \frac{j-k}{t} \geq \frac{1}{2} \alpha^2 t - C\,,
\]
whence the claim follows. 

\smallskip
{\it Step 3:(Compactness in space)}  We have
\begin{equation}\label{step3}
 \frac{d}{dt} \sum_j \big| u_{j+k} -u_j\big| \leq 0 \qquad \mbox{ for all } k \in \Z\,,
\end{equation}
 which follows from 
\[
 \begin{split}
 \frac{d}{dt} \sum_j&\big| u_{j+k} -u_j\big| = \sum_j \mbox{sgn}(u_{j+k}-u_j) \big( -u_{j+k}^2 + u_{j+k-1}^2 + u_j^2-u_{j-1}^2\big)\\
 &= \sum_j \mbox{sgn} (u_{j+k}-u_j) \big( (u_{j+k}+u_j)(u_j-u_{j+k}) + (u_{j+k-1}-u_{j-1})(u_{j+k-1}+u_{j-1})\big)\\
 & \leq -\sum_j |u_{j+k}-u_j| + \sum_j |u_{j+k-1}-u_{j-1}|=0\,.
 \end{split}
\]

\smallskip
{\it Step 4: (Compactness in time)} It holds
\begin{equation}\label{step4}
 \sum_j \big| \dot u_j\big| \leq \frac{2 }{{c_0}+t}\,.
\end{equation}
If we define $J_{+}:= \{ j \in \Z\,:\, \dot u_j > 0\}$, then by mass conservation $\sum_{j \in J_+} |\dot u_j| = \sum_{j \notin J_+} |\dot u_j|$. Furthermore 
 \eqref{step1} implies that $0\leq u_j-u_{j-1} \leq \frac{1}{{c_0}+t}$ for $j \notin J_+$. Hence $\sup_{j \notin J_+} |\dot u_j| = \sum_{j \notin J_+} (u_j+u_{j-1})(u_j-u_{j-1}) \leq \frac{C}{c_0+t}$.

\smallskip
{\it Step 5: (Tightness)} As in \cite{Ignatetal15} we conclude that $U_{\eps}(\tau,\cdot)$ is tight, more precisely that
\begin{equation}\label{step5}
 \int_{|x| \geq 2R} U_{\eps}(\tau,x)\,dx \leq \int_{|x|\geq R}U_{\eps}(0,x)\,dx + \frac{C \sqrt{t}}{R}\,.
\end{equation}
Indeed, if $\rho\colon \R \to [0,1]$ is a smooth cut-off function, with the properties that $\rho(x)=1$ for $x \geq 2$, $\rho(x)=0$ for $x\leq 1$ and $\rho'(x) \leq 1$, then we can estimate
\begin{align*}
 \frac{d}{dt} \int_{R}^{\infty} U_{\eps}(t,x) \rho\Big(\frac{x}{R}\Big)\,dx& = \int_{R}^{\infty} U_{\eps}(t,x)^2 \frac{\rho\big( \frac{x+\eps}{R}\big)
 -\rho\big(\frac{x}{R}\big)}{\eps}\,dx\\
 & \leq \frac{C}{R} \int_{R}^{\infty} U_{\eps}(t,x)^2\,dx 
\end{align*}
and \eqref{step5} thus follows from the bound \eqref{step2}.
\ignore{
As an alternative we give here another argument that implies that a limit of $U_{\eps}$ has compact support if one makes some stronger assumption on the initial data.
If we introduce $M(t):=\sum_j a^j u_j(t)$ for some $a>1$ and
assume in addition to \eqref{data} that $M(0) < \infty$, then we can use \eqref{step2} to compute
\[
 \frac{d}{dt} M(t) = \sum_j \big(a^{j+1}-a^j\big) u_j^2 \leq \frac{c_0 (a{-}1)M}{\sqrt{t}}\,,
\]
which implies
\begin{equation}\label{Mbound}
 M(t) \leq M(0) e^{2c_0(a-1)\sqrt{t}}= M(0) \alpha^{2 c_0 \frac{a-1}{\ln a} \sqrt{t}}\,.
\end{equation}
On the other hand
\[
 \sum_i a^i u_i(t)= \int \alpha^{\frac{x}{\eps}} U^{\eps} (\eps^2 t,x)\,dx
\]
and together with (\ref{Mbound}) we obtain
\[
 \int a^{\frac{x- 2 c_0 \frac{\alpha-1}{\ln \alpha}\sqrt{\tau}}{\eps}} U_{\eps}(\tau,x)\,dx \leq C\,.
\]
For $U:=\lim_{\eps \to 0} U_{\eps}$ this implies $U(\tau,x)=0$ for $x \geq 2 c_0 \frac{a-1}{\ln a}\sqrt{\tau}$.

If we assume for the initial data that $\sum_i e^{-i} u_i^0 < \infty$, which means in original coagulation variables that the total number of clusters
is initially finite, then we obtain directly (since the moment is decreasing) that $U = 0$ on $(-\infty,0)$.
}

\smallskip
{\it Step 6: (A priori estimates for $U_{\eps}$)}
The previous steps imply the following a-priori estimates for the rescaled function $U_{\eps}$. We have the uniform bound
$\|U_{\eps}(\tau,\cdot)\|_{L^{\infty}(\R)} \leq \frac{C}{\sqrt{\tau}}$
by \eqref{step2}, 
the entropy condition $\big( \partial_x U_{\eps}\big)_+ \leq \frac{1}{\tau}$ follows from \eqref{step1} and the uniform estimate on the time derivative
$\|\partial_{\tau}U_{\eps}(\tau,\cdot)\|_{L^1(\R)} \leq \frac{C}{\tau}$ is a consequence of \eqref{step4}. Moreover,  the equicontinuity 
\[
 \int_{\R} |U_{\eps}(\tau,x+h)-U_{\eps}(\tau,x)|\,dx \leq Ch
\]
follows from \eqref{step3}. Finally, \eqref{step5} provides tightness of $U_{\eps}$ locally uniformly in time.

\smallskip
{\it Step 7: (Compactness and limit equation)}
These estimates imply, using Riesz-Kolmogorov and Arzela-Ascoli, that the sequence $U_{\eps}$ is precompact in $C([\tau_1,\tau_2];L^1(\R))$ for
 arbitrary $0<\tau_1<\tau_2<\infty$. Thus, for a subsequence we have that $U_{\eps} \to U$ in  $C([\tau_1,\tau_2];L^1(\R))$ and it follows easily that
 $U$ is a weak solution of the Burgers equation. In addition it satisfies the same bounds as  $U_\eps$, in particular $\partial_x U \leq \frac{C}{t}$ which implies that $U$ is an entropy solution. The tightness estimate also implies $\int_{\R}U(\tau,x)\,dx=1$.
 
 \smallskip
 {\it Step 8: (Identification of the limit)}
Using  the equation for $u_j$, the estimate \eqref{step2} and mass conservation $\sum_j u_j=1$, we obtain the following weak continuity up to time $\tau=0$.
 For $\phi \in C^1_c(\R)$ 
 \[
 \Big| \int_{\R} \big( U_{\eps}(\tau+h,x)-U_{\eps}(\tau)\big) \phi(x)\,dx \Big| \leq C \|\phi'\|_{L^{\infty}} \sqrt{h} \qquad \mbox{ for all } \tau >0\,.
 \]
Hence, one can follow the lines of the elementary computations in Step II of the proof of Theorem 1.1 in \cite{Ignatetal15} to conclude that
$U(\tau,x) \to \delta_0$ as $\tau \to 0$. A key major ingredient to conclude the proof is  the uniqueness result for entropy solutions of the Burgers equation with 
initial data that are measures, provided in 
\cite{LiuPierre84}. It implies that  the limit $U$ is indeed the nonnegative $N$-wave with unit mass.
\end{proof}

 \begin{figure}[h!]
  \centering
  \includegraphics[width=.95\textwidth]{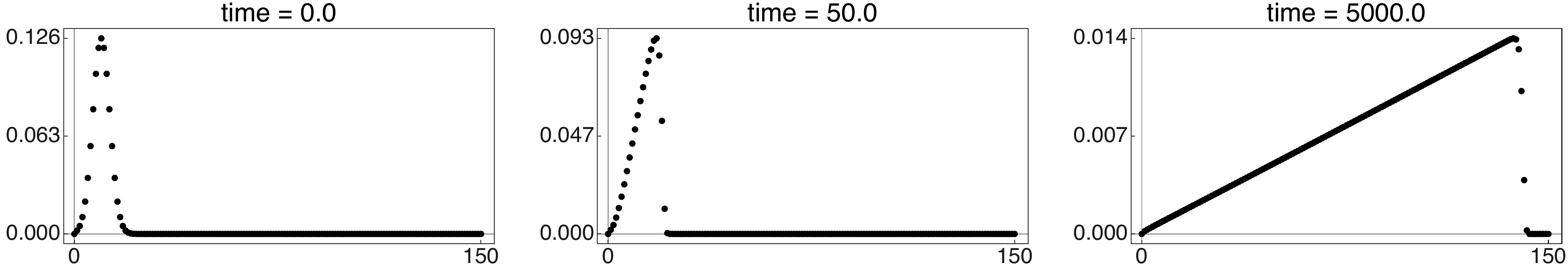}
\caption{ Numerical solution to the Burgers lattice \eqref{diagonal2} with integrable initial data: For larger times $t\gg1$,
the lattice data resemble an $N$-wave solution, see Proposition \ref{P.nwave}.}
  \label{fig:Nwavediagonal}
\end{figure}

 \paragraph{The coagulation equation \eqref{diagonal1}  as a family of lattices \eqref{diagonal2}.}
 As in \cite{LauNiVel16}, where the diagonal kernel with homogeneity smaller than one is considered,
 the analysis of  equation (\ref{diagonal1})  can be
reduced to the analysis of the  functions $\left\{ u(t,n+\theta ) \right\} _{n\in \mathbb{Z}}$ with $\theta \in \left[
0,1\right) .$ We will call each set of points $\left\{ X_{n}=n+\theta
:n\in \mathbb{Z}\right\}$  a fibre.  The union of all fibres with $\theta \in [ 0,1) $ covers the whole real line.
In spite of the fact that the dynamics given by independent fibres is
easy to describe, it yields interesting oscillatory behaviours and the onset
of peak-like solutions in some regions. We have seen that the $N$-waves
defined in \eqref{Ndef} 
depend on the mass contained in each fibre, i.e. on the number
$M(\theta) =\sum_{n=-\infty }^{\infty }u(t, n+\theta)$, which is constant in time for each fibre.
 Proposition \ref{P.nwave}
implies that the values $\left\{ u(t, n+\theta) \right\} _{n\in 
\mathbb{Z}}$ behave asymptotically as  an $N$-wave with mass $M(\theta)$,
that is
\begin{equation}
u( n+\theta ,t) \sim \frac{1}{\sqrt{t}}N\Big( \frac{n+\theta }{%
\sqrt{t}};M(\theta) \Big)  \qquad \mbox{ as } t\to \infty \,.
\label{A2}
\end{equation}%

Given that the support of the functions on the right-hand side of (\ref{A2})
depends on $\theta ,$ it follows that the function $u(t,X) $
increases linearly if $X\in \left[ 0,\sqrt{M_{\min }t}\right] ,$ where $%
M_{\min }=\min_{\theta \in \left[ 0,1\right] }M\left( \theta \right) .$ On
the other hand, if $X>\sqrt{M_{\min }t}$ we might have in each interval in
the $X$ variable with length $1$ regions where $u$ vanishes and regions
where $u$ is of order $\frac{1}{\sqrt{t}}.$ If we denote the integer part of
a real number $r\in \mathbb{R}$ as $\left\lfloor r\right\rfloor $ and the
fractional part as $\mbox{frac}\left( r\right) =r-\left\lfloor
r\right\rfloor $.\ Then (\ref{A2}) implies
\begin{equation}
u(t,X) \sim \frac{1}{\sqrt{t}}N\Big( \frac{X}{\sqrt{t}};M( 
\mbox{frac}(X) ) \Big) \qquad \mbox{ as } t \to
\infty\,.   \label{A3}
\end{equation}

Notice that this formula implies oscillatory behaviour in the variable $X$
for $u(t,X) $ if $X>\sqrt{M_{\min }t}$ in the generic case that  the function $M(\theta) $ is not constant (see Figure \ref{fig:uoscillations} for an illustration). 

 \begin{figure}[h!]
  \centering
  \includegraphics[width=.75\textwidth]{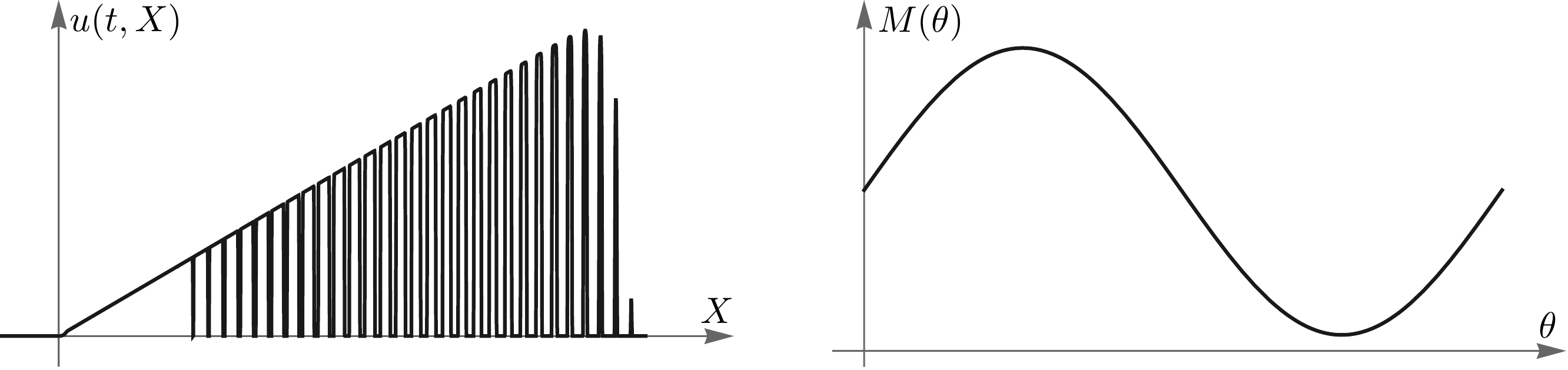}
\caption{ Cartoon of an highly oscillatory solution to the coagulation equation \eqref{diagonal1} that is composed by a $\theta$-family of N-wave solutions
to the lattice \eqref{diagonal2} with nonconstant mass distribution $M(\theta)$, see \eqref{A3}. }
  \label{fig:uoscillations}
\end{figure}

%
\subsection{The general case}\label{Ss.general}
%
 
 In this Section we compare the behaviour of
solutions of  \eqref{eq1} for kernels from the family $K_{\alpha}$ as in \eqref{kernelfamily}
 by means of formal asymptotics and numerical simulations. It turns out that the solutions of
\eqref{eq1} exhibit different features 
depending on whether $\alpha$ is small, i.e. $K_{\alpha}$ is close to the additive kernel, or $\alpha$ is large, i.e. $K_{\alpha}$
is close to the diagonal kernel.

A relevant feature of the case of large $\alpha $ is that the
homogeneous solutions are unstable, while they are stable for small $\alpha$ (with a threshold $\alpha_{{crit}}\approx 35$).

Also the shape of the traveling wave depends sensitively on the size of $\alpha$. Somewhat counterintuitively to the findings of the stability and instability respectively 
of the constant solution, for large $\alpha$ the numerical simulations suggest that the traveling wave is monotone, while for small $\alpha$ the traveling 
wave has oscillations on the left of the front, that
become large, when $\alpha $ becomes small. The threshold for oscillations to appear is $\alpha_{\ast} \approx 20.1$.

Finally we investigate in Section \ref{Ss.nwaves} the long-time behaviour of solutions with integrable data. For all values of $\alpha$ we observe
convergence to an $N$-wave. However, for small values of $\alpha$, there are strong oscillations on the left of the shock front (cf. the right picture in Figure \ref{fig:cartoon}).
This is explained by the expectation that the transition at the shock front is described by a rescaled traveling wave front which we have seen to be oscillatory for small $\alpha$.

\ignore{
      \begin{figure}[h!]
  \centering
  \includegraphics[width=.45\textwidth]{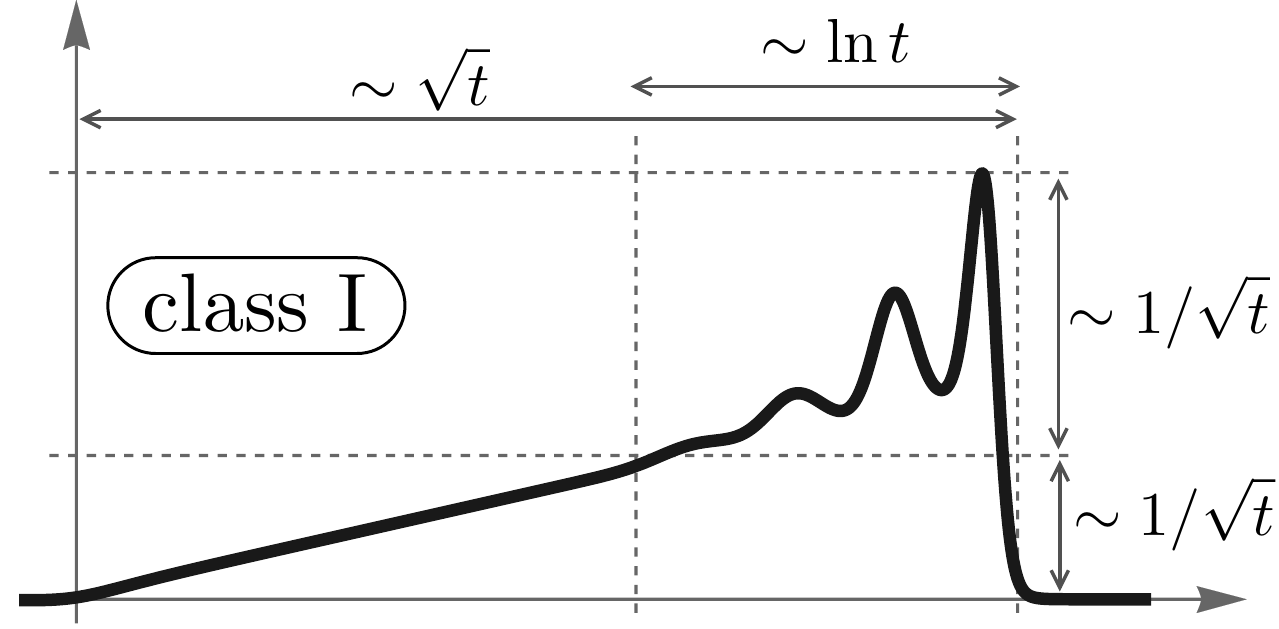}
\caption{Cartoon of large-time behaviour of solutions for kernels $K_{\alpha}$ with small $\alpha$.}
  \label{fig:cartoon1}
\end{figure}
 }
%
 \subsubsection{Instability of the constant solution for near-diagonal kernels}
 \label{Ss.stabilityneardiagonal}
%

Due to the property \eqref{finiteintegral} a constant is a solution of the coagulation equation \eqref{uequation}.
We examine the stability of this constant solution, that we can without loss of generality assume to be equal to $1$.
 If we plug the ansatz $u  =1+h(X)$ into \eqref{uequation} we obtain to leading order

 \begin{equation*}
\partial _{t} h(X)  +\partial _{X}\Big(
\int_{-\infty }^{0}dY\int_{\ln \big( 1-e^{Y}\big) }^{\infty }dZ\big[
K\big( e^{Y-Z},1\big) ( h( Y{+}X) +h( Z{+}X)) \big] \Big) =0\,.
\end{equation*}

We define the Fourier transform of $h$ by  $h( X) =\frac{1}{\sqrt{2\pi }}\int_{-\infty }^{\infty }H( k) e^{ikX}dk$ 
such that
\begin{equation}
\partial _{t}H(k) =M(k) H(k)\,,  \label{E0}
\end{equation}%
with
\begin{equation}
M( k) =-ik\int_{-\infty }^{0}dY\int_{\ln ( 1-e^{Y})
}^{\infty }dZ\Big[ K\big( e^{Y-Z},1\big) \big( e^{ikY}+e^{ikZ}\big) %
\Big]\,.  \label{E1}
\end{equation}
There are instabilities for particular wave numbers $k$ if $\mbox{Re}\big(M(k) \big) >0$. 

The following result  establishes the existence of unstable wave numbers
for a broad class of kernels.

\begin{prop}
\label{InstNearDiagonal}Suppose that $K$ is of the form
\begin{equation}
K(x,y) =( x+y) \eta \Big( \frac{x}{x+y}-\frac{1}{2}%
\Big) \,, \label{E4}
\end{equation}%
  where $\eta $ is a nonnegative, nonzero Radon measure with $\mathrm{supp}\,\eta (\cdot) \subset [ -\varepsilon ,\varepsilon]$, 
  $\eta (s) =\eta (-s) $ and  $\int_{\mathbb{R}%
}\eta (s) ds=1$.  Let $M(k) $ be as in (\ref{E1}) and $\eps>0$ sufficiently small. Then $\mbox{Re}\big( M\big( \frac{2\pi }{\ln  2 }\big) \big) >0$.
\end{prop}

\begin{proof}
We can write $\eta $ as $\eta (x) =\int_{-\frac{1}{2}}^{\frac{1}{2}}\eta (s)\delta (x-s) ds$.
Using (\ref{E1}) and (\ref{E4}) we obtain
\begin{equation}
M(k) =-\int_{-\frac{1}{2}}^{\frac{1}{2}}\eta(s)
W(k,s) ds\,,   \label{E6}
\end{equation}%
where
\begin{equation}
W(k,s) =ik\int_{-\infty }^{0}dY\int_{\ln ( 1-e^{Y})
}^{\infty }dZ\Big[ ( e^{Y-Z}+1) \delta \Big( \frac{e^{Y-Z}}{%
e^{Y-Z}+1}-\frac{1}{2}-s\Big) ( e^{ikY}+e^{ikZ}) \Big]\,.
\label{E5}
\end{equation}
The identity
$\big( e^{\xi }+1\big) \delta \Big( \frac{e^{\xi }}{e^{\xi }+1}-\frac{1}{2%
}-s\Big) =\frac{\big( e^{\xi }+1\big) ^{3}}{e^{\xi }}\delta \Big( \xi
-\ln \Big( \frac{1+2s}{1-2s}\Big) \Big)$ implies
\[
W(k,s) 
 =\frac{8ik}{( 1-2s)^{2}( 1+2s) }\int_{-\infty
}^{0}dY\int_{\ln ( 1-e^{Y}) }^{\infty }dZ
\Big[ \delta \Big(
Y{-}Z{-}\ln \Big( \frac{1+2s}{1-2s}\Big) \Big) (
e^{ikY}+e^{ikZ}) \Big]\,.
\]
With $\theta =\ln \big( \frac{1+2s}{1-2s}\big)$  the only
values of $Y$ contributing to the integral defining $W(k,s)$
are those with $Y\in [ -\ln ( 1+e^{-\theta }) ,0] $
and $Z=Y-\theta $ whence
\begin{align*}
W(k,s) & =\frac{8ik}{(1-2s)^{2}( 1+2s) }%
\int_{-\ln ( 1+e^{-\theta }) }^{0}( 1+e^{-ik\theta })
e^{ikY}dY \\
& =\frac{8}{(1-2s)^{2}( 1+2s) }\big( 1+e^{-ik\theta
}-( 1+e^{-\theta }) ^{-ik}-( 1+e^{\theta })^{-ik}\big)\\
&=\frac{8}{( 1-2s)^{2} (1+2s) }%
\Big( 1+\Big( \frac{1-2s}{1+2s}\Big) ^{ik}-\Big( \frac{1+2s}{2}\Big)
^{ik}-\Big( \frac{1-2s}{2}\Big) ^{ik}\Big)\,.
\end{align*}

We take $k=\frac{2\pi }{\ln 2 }.$ Then, combining the factors 
$\ln  2 $ and using the periodicity of $\cos$, we find
\begin{align*}
 \mbox{Re}\, W\big( k_{1},\tfrac{1}{2}+s\big)  
& =\frac{8}{( 1-2s) ^{2}(1+2s) }\Big( 1+\cos \Big(
k_{1}\ln \Big( \frac{1-2s}{1+2s}\Big) \Big) \\
&\qquad \quad -\cos ( k_{1}\ln
( 1+2s)) -\cos ( k_{1}\ln ( 1-2s) \Big)\Big)\,.
\end{align*}

Expanding the function above using Taylor's series we 
find 
$\mbox{Re}\left( W\left( k_{1},\frac{1}{2}+s\right) \right) =-32\left(
k_{1}s\right) ^{2}+O\left( s^{3}\right)$ as $s\rightarrow 0$, 
and thus  the result follows from \eqref{E6} if $\varepsilon $ is sufficiently
small.
\end{proof}

\paragraph{PDE approximation and convective instabilities}
It is possible to give an interpretation of this instability of the constant
by approximating the coagulation
equation \eqref{uequation} by a PDE.
 Equation \eqref{uequation} with $K$ as in \eqref{E4}  can be rewritten as
\begin{align*}
\partial _{t} u(X)& +u(X)\int_{-\infty }^{\infty }\Big( \frac{e^{Z}+1}{e^{Z}}\Big) \eta \Big( 
\frac{1}{1+e^{Z}}-\frac{1}{2}\Big) u( X+Z) dZ \\
&-\int_{-\infty }^{0}\frac{1}{ (1-e^{Y})^{2}}\eta \Big( e^{Y}-%
\frac{1}{2}\Big) u( X+Y) u( X+\ln ( 1-e^{Y})) dY =0\,.
\end{align*}

We now use the change of variables $\frac{1}{1+e^{Z}}-\frac{1}{2}=\frac{\xi 
}{2}$ in the first integral and $e^{Y}-\frac{1}{2}=\frac{\xi }{2}$ in the
second one. Moreover, we will assume, in order to simplify the numerical
constants, that $\eta (s) =\frac{1}{4\varepsilon }\zeta \left( \frac{2s}{%
\varepsilon }\right)$, $\int \zeta (s) ds=1$ and $ \zeta
(s) =\zeta (-s)$.
Then
\begin{align}
\partial _{t}u(X)& +u(X)\int_{-1}^{1}u\Big( X+\ln\Big ( \frac{1-\xi }{1+\xi }\Big)\Big) 
\frac{\frac{1}{\varepsilon }\zeta \big( \frac{\xi }{\varepsilon }\big)
d\xi }{( 1+\xi )( 1-\xi ) ^{2}}  \nonumber \\
&-\int_{-1}^{1}u( X-\ln  2 +\ln ( 1+\xi )) u( X-\ln 2 +\ln ( 1-\xi)) 
\frac{\frac{1}{\varepsilon }\zeta \big( \frac{\xi }{\varepsilon }\big)
d\xi }{( 1+\xi) ( 1-\xi ) ^{2}}  =0 \,. \label{E8}
\end{align}

We can expand the terms with logarithm in the arguments, using Taylor up to second order. Then  (\ref{E8}) can be
approximated by
\begin{align}
0=&\partial _{t}u(X) +\Lambda _{0}\Big[ u(X)^{2}-u( X-\ln 2)^{2}\Big]  \nonumber \\
&+\Lambda _{2}\varepsilon ^{2}\big[ -2u(X) \partial_{X}u(X) +2u(X) \partial _{X}^{2}u(X)
-W(X) \big]   \label{E9}
\end{align}%
with
\[W( X) =-u( X{-}\ln  2 ) \partial
_{X}u( X{-}\ln  2 ) +u( X{-}\ln 2) \partial _{X}^{2}u( X{-}\ln 2)-\big( \partial _{X}u( X{-}\ln 2) \big) ^{2}
\]and 
\begin{equation*}
\Lambda _{0}=\int_{-1}^{1}\eta \frac{\zeta(\eta) d\eta }{%
( 1+\varepsilon \eta) ( 1-\varepsilon \eta )^{2}}\,,\qquad  \Lambda _{2}=\int_{-1}^{1}\eta ^{2}\zeta (\eta) d\eta\,.
\end{equation*}

Notice that $\Lambda _{0}\rightarrow 1$ as $\varepsilon \rightarrow 0.$ The
approximation (\ref{E9}) is valid as long as the characteristic lengths
associated to the function $u$ are smaller than $\varepsilon .$ This
condition holds, for instance, if $\vert \partial _{X}u\vert \ll 
\frac{u}{\varepsilon }$ and similar conditions for higher order
derivatives hold. Equation (\ref{E9}) suggests that the coagulation equation
is ill-posed and has the same type of instabilities as backward parabolic equations. However, this is not really so, because the
approximation (\ref{E9}) is only valid if the wave numbers $k$ are smaller
than $\frac{1}{\varepsilon }.$ Nevertheless the type of instabilities
exhibited by backward parabolic equations explain the instabilities
 in Proposition \ref{InstNearDiagonal}. Indeed, writing $%
u(X) =1+h(X) $ and linearizing (\ref{E9}) around $%
u=1, $ we obtain
\begin{align*}
\partial _{t}h(X)& +2\Lambda _{0}\big[ h(X) -h( X-\ln 2) \big]  \\
&+\Lambda _{2}\varepsilon ^{2}\big[ -2\partial _{X}h( X)
+2\partial _{X}^{2}h(X) +\partial _{X}h( X-\ln 2) -\partial _{X}^{2}h( X-\ln 2) %
\big] =0\,.
\end{align*}

Then, the Fourier transform of $h,$ denoted as $H(k)$, satisfies
(\ref{E0}) with $M(k)$ approximated as
\begin{equation}
M( k) =-2\Lambda _{0}( 1-e^{-ik\ln 2 })
-\Lambda _{2}\varepsilon ^{2}\big[ -2ik-2k^{2}+ike^{-ik\ln  2
}+k^{2}e^{-ik\ln 2}\big]    \label{F1}
\end{equation}%
if $k\ll \frac{1}{\varepsilon }$ and thus
\begin{equation*}
\mbox{Re}( M(k)) =-2\Lambda _{0}( 1-\cos (k\ln 2 ) ) +\Lambda _{2}\varepsilon ^{2}\big[
2k^{2}+k\sin ( k\ln 2) -k^{2}\cos( k\ln  2) \big]
\end{equation*}%
if $k\ll \frac{1}{\varepsilon }$. The terms associated to the
diagonal kernel yield the periodic contribution $-2\Lambda _{0}( 1-\cos( k\ln 2)) $ which gives stable behaviour,
although neutrally stable for $k=k_{n}=\frac{2\pi n}{\ln  2 }$
with $n\in \mathbb{Z}$. The leading contribution among the
terms of order $\varepsilon $ for large $k$ is the term $2\Lambda
_{2}\varepsilon ^{2}k^{2}$ which is due to the backward parabolic term $%
2\Lambda _{2}\varepsilon ^{2}\partial _{X}^{2}h(X)$. The
instability induced by this term at the values $k=k_{n}$ explain the
instabilities obtained above. This term becomes of order one if $k$ is of
order $\frac{1}{\varepsilon }.$

The linear instability of the homogeneous positive solutions described above
is analogous to many instabilities arising in problems of pattern formation
(see e.g. \cite{Turing} or page 16 in  \cite{CrossHohenberg}), but it does not take place for the
standard viscous regularization of the Burgers equation.

If $\bar{k}\in \R$ is one of the values for which $\mbox{Re}( M_{\varepsilon }(k)) >0$ we obtain disturbances
of homogeneous solutions of the form
\begin{equation}
\exp \big( \mbox{Re}( M( \bar{k})) t\big) \exp
( i( \bar{k}X+\mbox{Im}( M( \bar{k}))t) )\,.  \label{A1}
\end{equation}
Notice that the approximation (\ref{F1}) allows to approximate $\mbox{Im}( M(k)) $ if $k$ is of order one since 
$\mbox{Im}( M(k)) =-2\Lambda _{0}\sin ( k\ln 2 )$
as $\varepsilon \rightarrow 0.$ Therefore (\ref{A1}) can be interpreted as a
disturbance with wave number $\bar{k}$ propagating towards increasing values
of $X$ with velocity $\frac{\vert \mbox{Im}(M(\bar{k}))\vert }{\bar{k}}.$ 
Notice that
due to the presence of convective terms a small disturbance that is initially localized at 
$X=X_{0}$ with unstable wave numbers can become of
order one, due to its exponential growth, at values $X=X_{1}$ with $%
X_{1}-X_{0}\gg 1.$ This phenomenon, known as convective instability takes
place in many other situations, such as spiral waves in excitable media \cite{SandScheel} or plasma physics (see e.g. Chapter 62 of \cite{LL}).

%
\subsubsection{Stability of the constant solutions for $K_{\alpha}$.}\label{Ss.stabilityalpha}
%
Interestingly, for kernels that are not close to the diagonal one, the constant solution is stable.
Here we consider  $K_{\alpha}(x,y) $ as in \eqref{kernelfamily} for different values of $\alpha$ with the normalization
\begin{equation}\label{calphadef}
 c_{\alpha}= \frac{1}{B(\alpha,\alpha{-}1)\big(\psi(2\alpha{-}1)-\psi(\alpha)\big)}\,, \qquad \mbox{ where } \psi(z)=\frac{\Gamma'(z)}{\Gamma(z)}\,, 
\quad B(\alpha,\alpha{-}1)=\frac{\Gamma(\alpha)\Gamma(\alpha{-}1)}{\Gamma(2\alpha{-}1)}\,.
\end{equation}
This choice is such that the constant $A$ in \eqref{finiteintegral} satisfies $A=1$.

We will denote as $M_{\alpha }(k) $
the function $M(k) $ defined in \eqref{E1} for the kernels $K_{\alpha}(x,y)$. Some elementary, but tedious computations yield
\begin{align}\label{B4}
{M}_{\alpha }(k)  &=-\frac{\Gamma ( 2\alpha {-}1)}{\Gamma (\alpha) ( \psi ( 2\alpha {-}1) -\psi
( \alpha ) ) }\left[ \frac{\Gamma (\alpha) }{\Gamma ( 2\alpha {-}1) }-\frac{\Gamma ( \alpha +ik) }{%
\Gamma ( 2\alpha {+}ik{-}1) }\right]  \\
&\qquad -\frac{\Gamma ( 2\alpha {-}1) \Gamma ( \alpha {+}ik{-}1) }{\Gamma (\alpha) \Gamma ( \alpha {-}1) 
( \psi ( 2\alpha {-}1) -\psi( \alpha ) ) }\Big[ \frac{%
\Gamma ( \alpha {-}ik) }{\Gamma ( 2\alpha {-}1) }-\frac{%
\Gamma ( \alpha) }{\Gamma ( 2\alpha {+}ik{-}1) }\Big] \nonumber
\end{align}

We have plotted the function $k\rightarrow \mbox{Re}\left( M_{\alpha }(k) \right) $ for $k\in \mathbb{R}$ for different values of $\alpha$
in Figure \ref{fig:bifurc2}.

 \begin{figure}[h!]
  \centering
  \includegraphics[width=.95\textwidth]{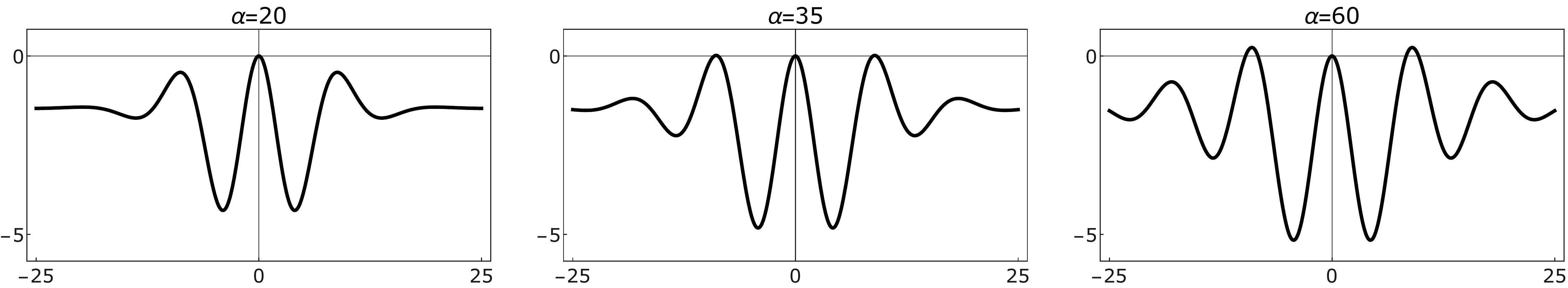}
 \caption{
Real part of the function ${M}_\alpha$ from \eqref{B4} for three values of $\alpha$. For $\alpha=\alpha_{\mathrm{crit}}\approx35$ we observe a change in the stability of the constant solution, see \eqref{C2}.
 } %
  \label{fig:bifurc2}
\end{figure}

Notice that for large $\alpha $  
we are  in an analogous situation  to the one discussed in
 Section \ref{Ss.stabilityneardiagonal} and we can expect 
to find values $k\in \mathbb{R}$ with $\mbox{Re}(M_{\alpha}(k)) >0$. This can be seen in Figure \ref{fig:bifurc2}.
On the other hand, we can also see  that $\mbox{Re}(M_{\alpha }(k)) \leq 0$ for smaller values of $\alpha$.
Moreover $\mbox{Re}( M_{\alpha }(k)) =0$ only for $%
k=0.$ Therefore, the constant solutions are stable for this range of values
of $\alpha .$ The computations of $\mbox{Re}(M_{\alpha }(k)) $ indicate that the critical value of $\alpha $ for which
the change of stability takes place is
$\alpha _{\mathrm{crit}}=35$.

     In Figure \ref{fig:const-stab}  we see the results of numerical simulations of the coagulation equation with initial data that
     are perturbations of the constant solutions. They confirm that for  $\alpha<\alpha_{\mathrm{crit}}$ the perturbations do not grow, while for
     $\alpha>\alpha_{\mathrm{crit}}$ the perturbation becomes
     oscillatory and grows. 

 \begin{figure}[h!]
  \centering
  \includegraphics[width=.98\textwidth]{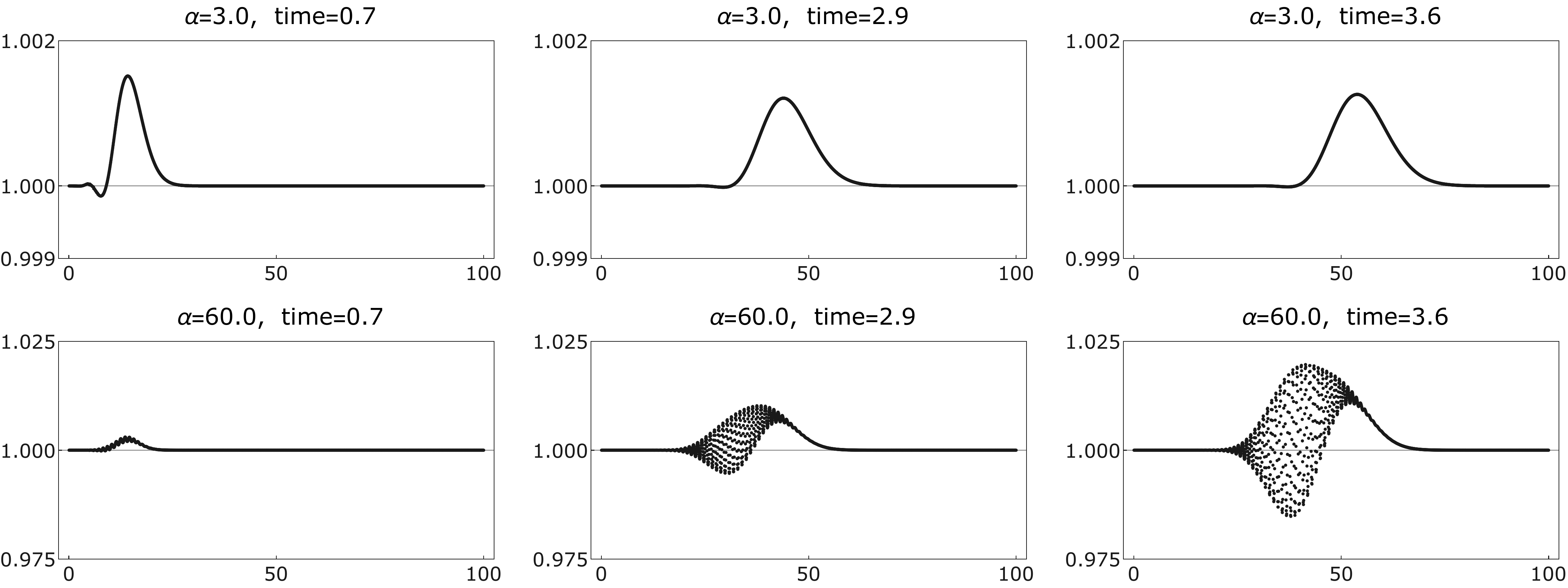}
\caption{Numerical simulations of the coagulation equation with initial data that are perturbations of a constant for $\alpha=3.0$ (top) and
$\alpha=60$ (bottom).The numerical scheme as well as the choice of the space and time units are described in the appendix.}
  \label{fig:const-stab}
\end{figure}

 \subsubsection{Traveling wave solutions.}
\label{Ss.traveling}

 In this section we compare the traveling wave solutions which solve \eqref{Gequation} with the kernels $K_{\alpha }(x,y) $ 
 in \eqref{kernelfamily} for different values $\alpha$.
As mentioned before, we can set without loss of generality $b=1$. Then, since $A=1$ we have 
 $G(-\infty) =1$.  Numerical
computations of these traveling waves for different values of $\alpha $ can
be seen in Figure \ref{fig:twaves}.

   \begin{figure}[h!]
  \centering
  \includegraphics[width=.94\textwidth]{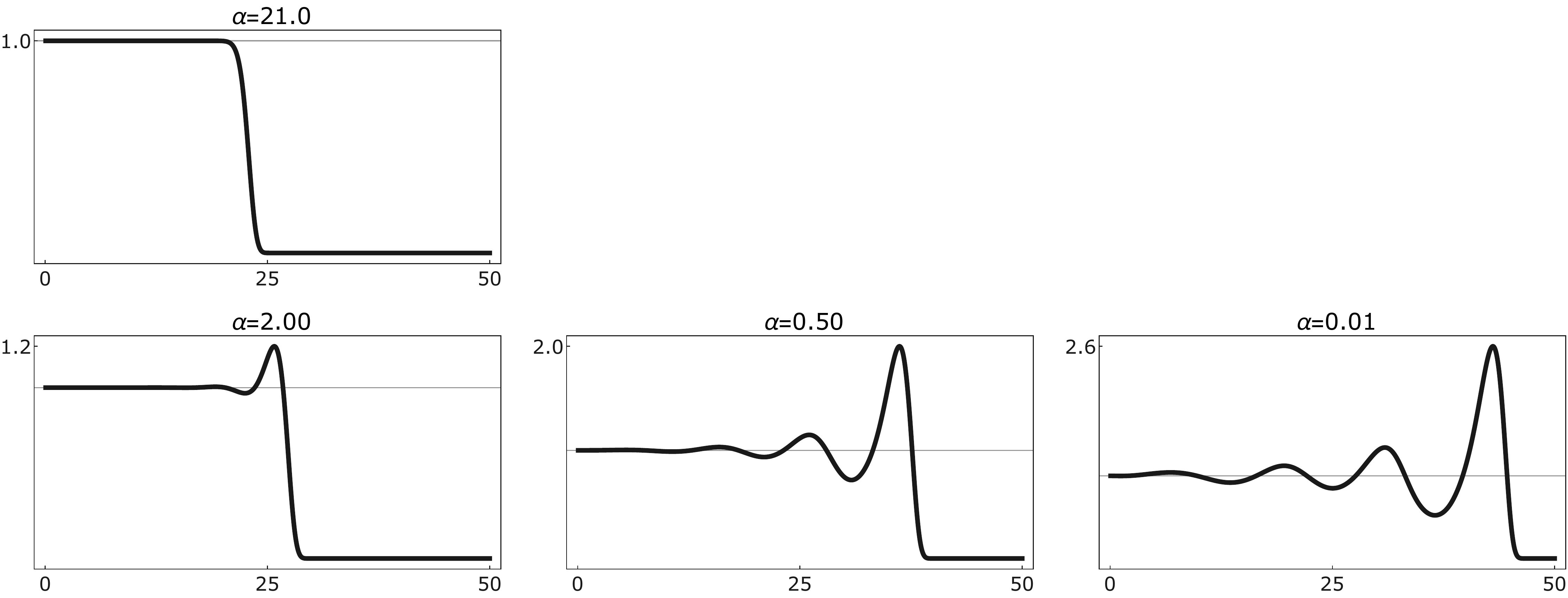}
  \caption{The shape of the traveling wave for different values of $\alpha$ ($u$ against $X$).}
  \label{fig:twaves}
\end{figure}

These pictures show
that the waves separate from the value $G=1$ at $X=-\infty $ in an
oscillatory manner. In order to understand this fact we consider the
linearization of \eqref{Gequation} near $G(-{\infty })=1$ and write $%
G=1+H(X)$. 
Then we obtain
the following linearized problem
\begin{equation}
H(X) =\int_{-\infty }^{0}dY\int_{\ln \left( 1-e^{Y}\right)
}^{\infty }dZK\left( e^{Y-Z},1\right) \left[ H( Y+X) +H(Z+X) \right] \,. \label{LinTW}
\end{equation}
If we look for solutions of \eqref{LinTW} of the form $H(X)=e^{i\eta X}$ with $\eta \in \mathbb{C}$
we find 
$-i\eta =M( \eta )$
with $M(\cdot)$ as in \eqref{E1}.

In order to obtain solutions of \eqref{LinTW} which tend to zero as $%
X\rightarrow -\infty $ we need to obtain solutions of $-i\eta =M( \eta )$
such that $\mbox{Im}(\eta) <0.$ Oscillatory behaviours  arise if the corresponding solution  satisfies $%
\mbox{Re}(\eta ) \neq 0.$ Therefore, we might expect to have
solutions of \eqref{Gequation} oscillating as $X\rightarrow -\infty $ if
the roots of $-i\eta=M(\eta)$ with the largest value of $\mbox{Im}(\eta ) $ in the half-plane $\left\{ \eta :\mbox{Im}(\eta)
<0\right\} $ satisfy $\mbox{Re}( \eta) \neq 0.$
Thus, if $M_{\alpha}$ denotes the function $M(\cdot)$ for the kernels $K_{\alpha}$ in \eqref{kernelfamily} with $c_{\alpha}$ as in \eqref{calphadef}
we need to investigate the roots of 
\begin{equation}
{M}_{\alpha }( k) +ik=0  \label{B5}
\end{equation}
in the half plane 
 $\{ \mbox{Im}( k) <0\}$. These roots are
plotted in Figure \ref{fig:bifurc1}. We can
see  that, for $\alpha <\alpha _{\ast },$ with $\alpha_{\mathrm{crit}}\approx 20,$ the roots of \eqref{B5} in the half-plane $\left\{ \mbox{%
Im}(k) <0\right\} $ with largest value of $\mbox{Im}(k) $ have $\mbox{Re}(k) \neq 0$, while
for  $\alpha>\alpha_{\ast}$   a unique root 
with $\mbox{Re}(k) =0$. 

 \begin{figure}[h!]
  \centering
  \includegraphics[width=.9\textwidth]{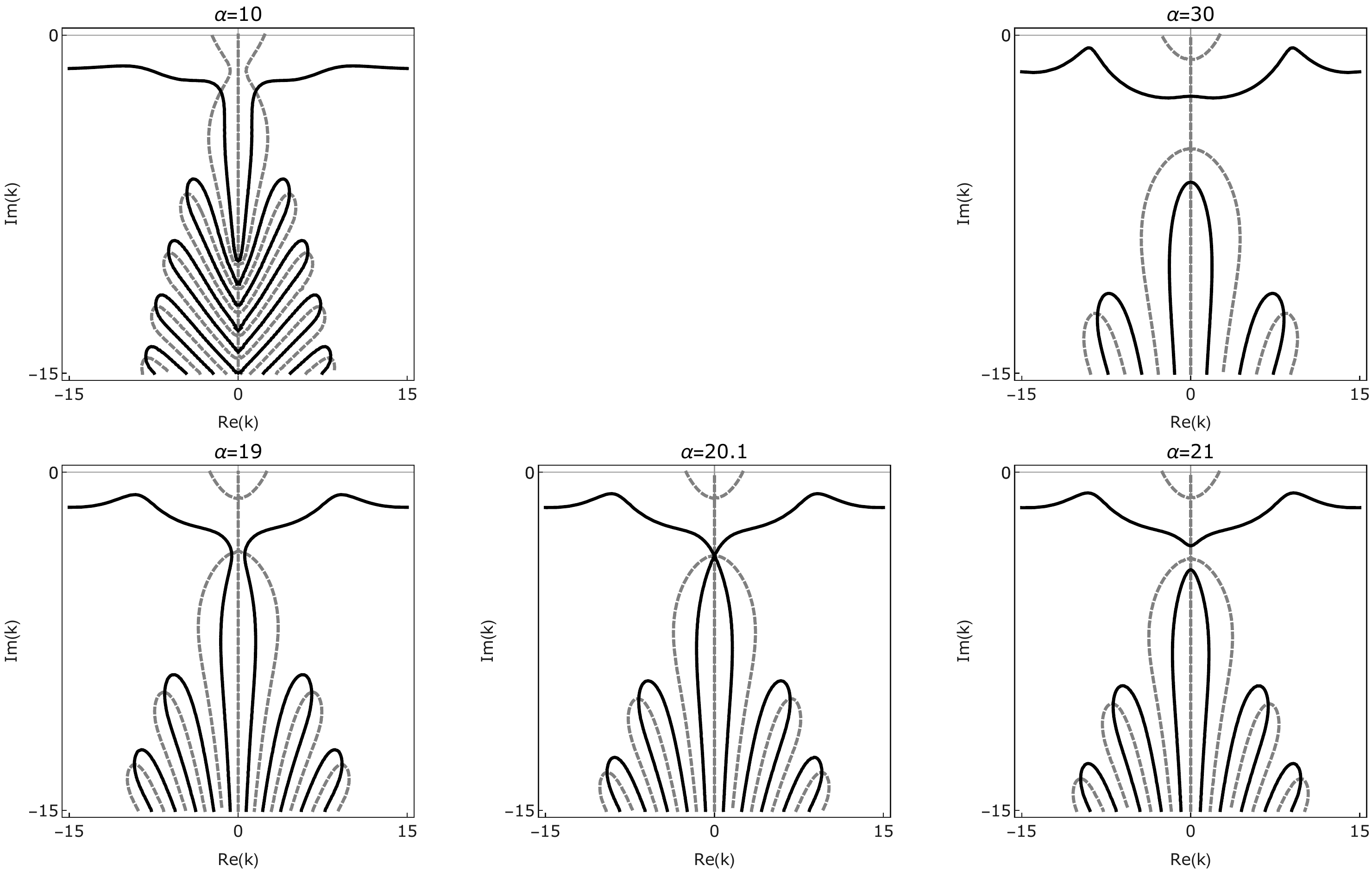}
 \caption{Zeros of 
$\mathrm{Re}({M}_\alpha(k)+i k)$ (black, solid)
 and $\mathrm{Im}({M}_\alpha(k)+ik)$ (gray,dashed) in the complex half plane $\mathrm{Im}(k)<0$ and for several values of
 $\alpha$, see \eqref{B5}. The traveling wave solution for the $\alpha$-kernel \eqref{kernelfamily} is oscillatory and monotone for 
$\alpha<\alpha_{\ast}$ and $\alpha>\alpha_{\ast}$, respectively, where $\alpha_{\ast} \approx 20.1$.
} %
  \label{fig:bifurc1}
\end{figure}

Therefore
 we  expect that $G(X) \rightarrow 1$ as $X\rightarrow -\infty ,$ with
oscillations of decreasing  amplitude as $X\rightarrow -\infty $ if $%
\alpha <\alpha _{\ast }$, while $G$ is monotone if  $\alpha >\alpha _{\ast }$.
 This scenario is confirmed by direct numerical
simulations of the solutions of \eqref{uequation} (cf. Figure \ref{fig:twaves}).

Interestingly, since the stability of the constant solution is equivalent to 
\begin{equation}
\max_{k\in \mathbb{R}}\big( \mbox{Re}( {M}_{\alpha }(k) ) \big) \leq 0\,,  \label{C2}
\end{equation}%
we see that both, stability of the constant and monotonicity of traveling waves, depend on the same analytic function $M_{\alpha}(\cdot)$. 
These conditions are obviously not equivalent. 
If $\alpha <\alpha_{\ast}$ the constant solution is stable, but the traveling wave is oscillatory, while,
 a bit paradoxically,  for $\alpha >\alpha _{\mathrm{crit}}$ the constant solution is unstable for the traveling wave is monotone for $X \to-\infty$.
Only if $\alpha \in(\alpha _{\ast },\alpha _{\mathrm{crit}}) $  the constant solution is stable and the traveling is monotone.

Due to the convective instabilities discussed in Section \ref{Ss.stabilityneardiagonal}
structures such as 
traveling waves might only be stable under  perturbations that are sufficiently small
as $X\rightarrow -\infty $ such that they  do not have time to increase before they arrive
at the front of the wave.
On the other hand the dissipative effects at
the front where the values of $u$ decay in lengths of order one, might have
stabilizing effects. Therefore, the stability of the
fronts for the diagonal kernel suggests that the fronts should be stable
also for near-diagonal kernels under perturbations which are still small
when they arrive to the back of the front.

Figure \ref{fig:tw-stab} shows the result of numerical simulations which indicate the stability of the traveling wave for small $\alpha$, but shows that the traveling wave
is unstable for $\alpha>\alpha_{\mathrm{crit}}$. 
\begin{figure}[h!]
  \centering
  \includegraphics[width=.95\textwidth]{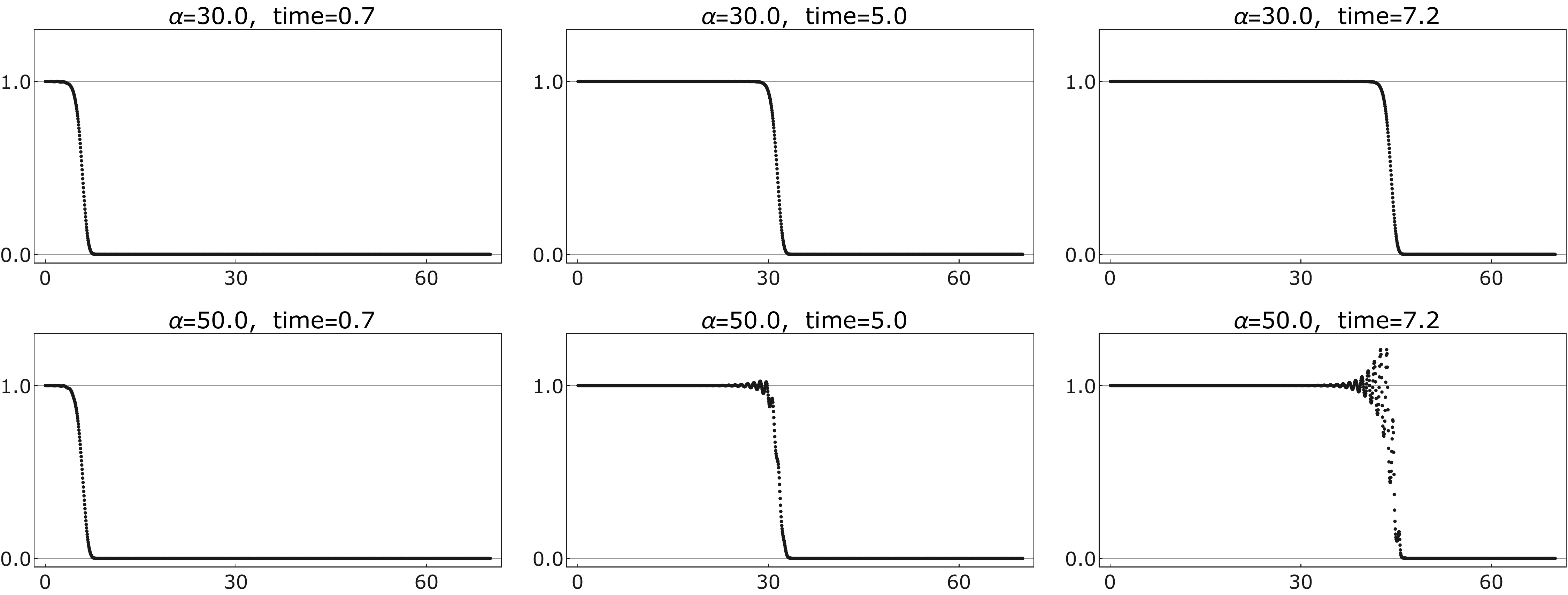}
\caption{Stability versus instability of the traveling wave for small and large $\alpha$ respectively.}
  \label{fig:tw-stab}
\end{figure}

In the range $\alpha \in ( 0,\alpha _{\ast }) $ there are no
 convective instabilities, but the traveling waves connecting the values of $G(-\infty) >0$ with $G( \infty ) =0$ exhibit strong
oscillations and peaks. The numerical simulations of \eqref{uequation}
suggest that these waves are nevertheless stable.
This is an intriguing feature which deserves a more careful
understanding. 
In a forthcoming paper \cite{NV16} we will  construct via formal matched asymptotic expansions a traveling wave solution for kernels 
that are similar to $K_{\alpha}$ for small $\alpha$. This analysis  also  yields precise expressions for the size and width of the peaks.

The origin of these oscillations is not clear to us. We have however observed them in other examples, in the  construction of self-similar solutions
to the coagulation equation with kernel $K(\xi,\eta)=(\xi\eta)^{\lambda}$ with $\lambda \in (0,1/2)$ \cite{MNV11}. 
Also in this case solutions develop oscillations that become more extreme the smaller $\lambda$ is. 
Oscillatory traveling waves of similar shape  and their stability properties have also been studied in \cite{PeSmWe93}, for a generalized KdV-Burgers equation,
which contains
diffusive and dispersive effects. The effect of the diffusive
effects is strong enough to allow the existence of stable traveling waves,
but the dispersive effects are relevant yielding highly oscillatory
traveling waves. It is natural to ask if the oscillatory behaviour of the traveling waves 
in the coagulation equation can be explained by  a similar competition between
diffusive and dispersive effects.

  \subsubsection{Long-time behaviour of solutions for integrable data.}
  \label{Ss.nwaves}
  
  \begin{figure}[ht!]
  \centering
  \includegraphics[width=.95\textwidth]{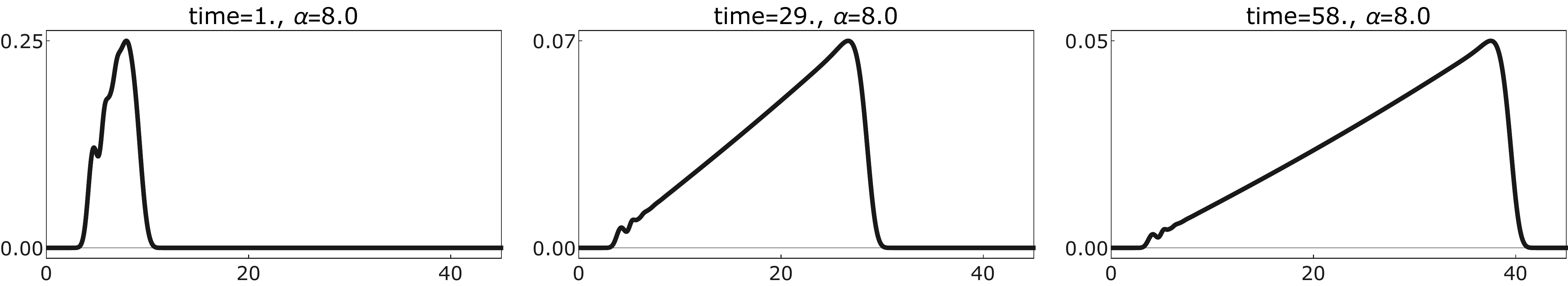}\\
  \includegraphics[width=.95\textwidth]{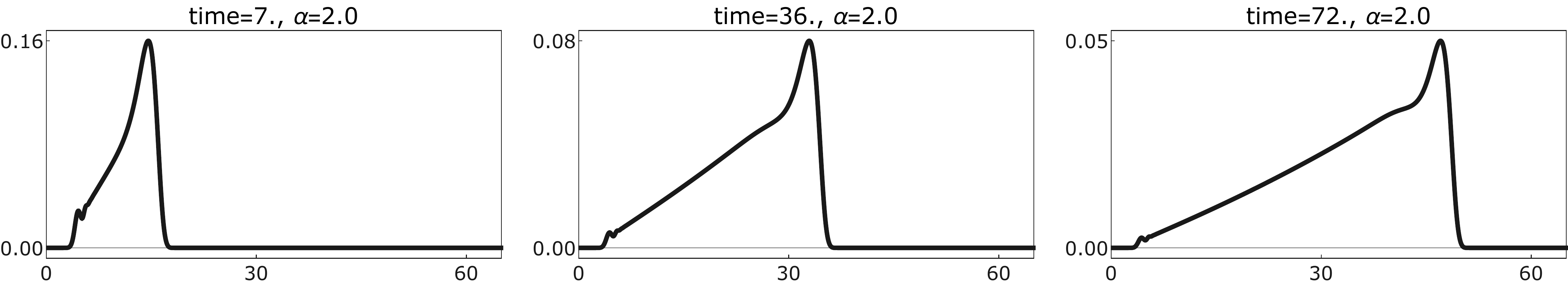}\\
  \includegraphics[width=.95\textwidth]{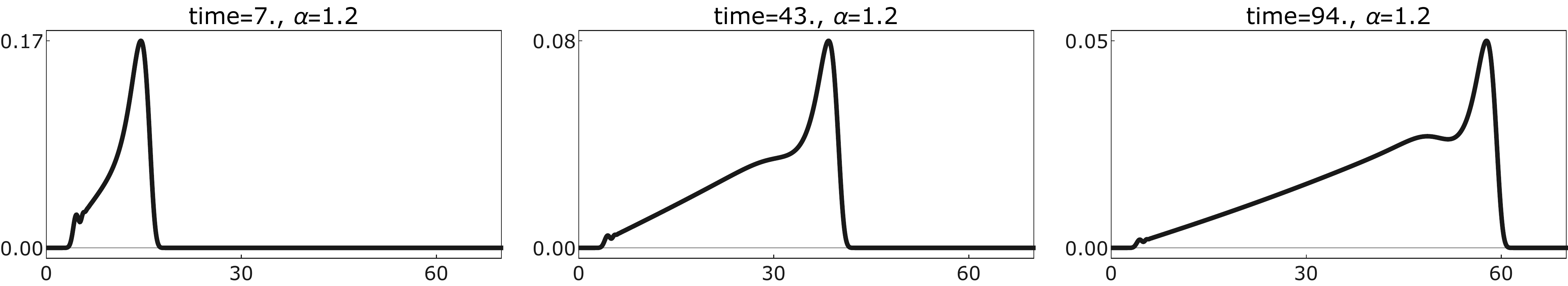}\\
  \includegraphics[width=.95\textwidth]{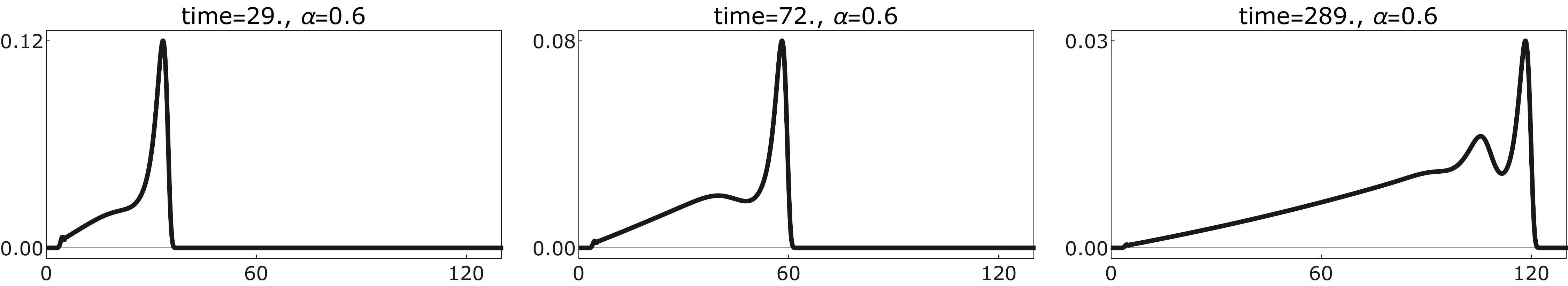}
  \caption{Convergence to the N-wave for $\alpha=8$ (top),  $\alpha=2$ (second row),
  $\alpha=1.2$ (third row) and $\alpha=0.6$ (bottom).}
  \label{fig:Nwaves}
\end{figure}

 In the case of solutions of \eqref{uequation} with integrable initial
data, the analogies  with the classical Burgers
equation suggest that  solutions of \eqref{uequation} should behave
asymptotically as $N$-waves. Van Dongen and Ernst \cite{vanDoErnst88} had already predicted the correct time scale on which a non-trivial limit should appear and also predicted that
the transition profile is given as a solution of \eqref{Gequation}, but the explicit connection to the Burgers equation and $N$-waves has not been made.

In numerical simulations for $K_{\alpha}$ with $\alpha<\alpha_{\mathrm{crit}}$  (see Figure \ref{fig:Nwaves}) we indeed observe convergence to an $N$-wave.
We also see, however, 
that the approach to the $N$-wave is quite unusual, at least for small $\alpha$. In this case, as
 discussed in Section \ref{Ss.stabilityalpha} the
traveling waves connecting the values at the back and rear of the shock
appearing in the $N$-wave exhibit strong oscillations if $\alpha $ is small
and have the property that the maximum values of $u$ along these waves are
much larger than the values of $u$ on the back of the wave. These strong
oscillations are visible in Figure \ref{fig:Nwaves}. 
We also conjecture that, as predicted in the case of diagonal kernel by \cite{BenNaimKrap12}, 
that the transition region scales as $\ln t$ as $ t \to \infty$ (see right panel in Figure \ref{fig:cartoon}).

In the case that $\alpha>\alpha_{\mathrm{crit}}$, the same considerations indicated above concerning the stability of the
fronts under convective instabilities can be raised about the stability of
the $N$-waves. In the region where the function $u$ is increasing along the $%
N$-wave we can assume that the amplitude is approximately constant and then
disturbances with wavelength of order one might propagate along the wave and
modify in a significant manner the shape of the wave before the disturbances
arrive to the region where $u$ decreases. Numerical simulations for smooth data with compact support for large $\alpha$
indicate that this is indeed the case (see Figure \ref{fig:nwave-stab}). The main issue here is to estimate
the amplitude of the disturbances associated to unstable Fourier wave numbers
in regions where $X$ is of order one. 
We also observe in Figure \ref{fig:nwave-stab} oscillations at the shock which are reminiscent of the oscillations for the diagonal kernel
(see Figure \ref{fig:uoscillations}) that are due to fibres with different mass. Thus, it will also be
relevant to understand the mass exchange between different fibres that takes place for kernels that are close to the diagonal
one.

 \begin{figure}[h!]
  \centering
  \includegraphics[width=.95\textwidth]{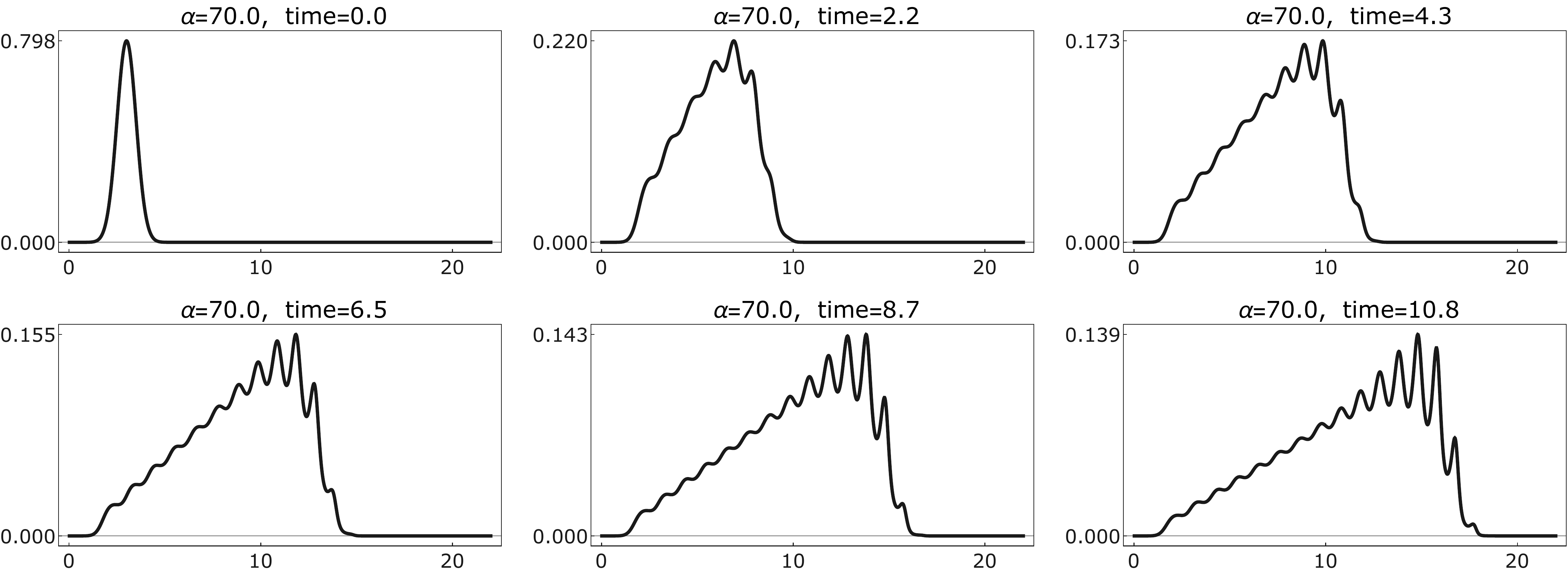}
  \caption{Evolution for smooth integrable initial data for  $\alpha=70$.}
  \label{fig:nwave-stab}
\end{figure}

%
\section{Summary and concluding remarks}
%
In this article we studied the long-time behaviour of solutions to the coagulation equations for kernels with homogeneity one.
For kernels with this homogeneity the long-time behaviour depends on whether the kernel is diagonally dominant (called class-I in the literature) or not.
In the latter case one expects that a family of self-similar solutions with finite mass exists and solutions converge to a member of this family
which has the same decay behaviour as the initial data. However, to  rigorously 
establish just the existence of such a family and in particular to determine their precise range seems in general a difficult task.

For class-I kernels, no self-similar solutions with finite mass can exist. In suitable variables (see \eqref{newvariables} and \eqref{uequation})
heuristics suggest that the long-time behaviour should be as in the classical Burgers equation, which means that for integrable
data solutions approximate an $N$-wave in the limit. We performed numerical simulations that 
confirm this conjecture, but they also reveal that the details of how the $N$-wave is approximated is quite unusual. In particular, for kernels that 
are not close to the diagonal one, we observe strong oscillations near the shock front. Those can be explained by the fact that the transition at the shock
is given by traveling wave profiles. A linear analysis indeed suggests that these traveling waves are oscillatory which is also observed in  numerical simulations 
for Riemann data.
It would be very interesting to rigorously prove the existence of such traveling waves and understand their regularity and stability properties.
In a first step this might be feasible for kernels close to the diagonal one. This case is also of particular interest due to the instability of the constant
solutions that we found in this regime and the results of numerical simulations that suggest that also traveling waves are unstable for such kernels.

\ignore{
The Cartoon \ref{fig:cartoon} summarizes  the expected long-time behaviour of solutions to the coagulation equation in the variables \eqref{newvariables}, \eqref{uequation} for the
two different classes of kernels in the case that the class-I kernel is not too  close to the diagonal one. 

\begin{figure}[h!]
  \centering
  \includegraphics[width=.4\textwidth]{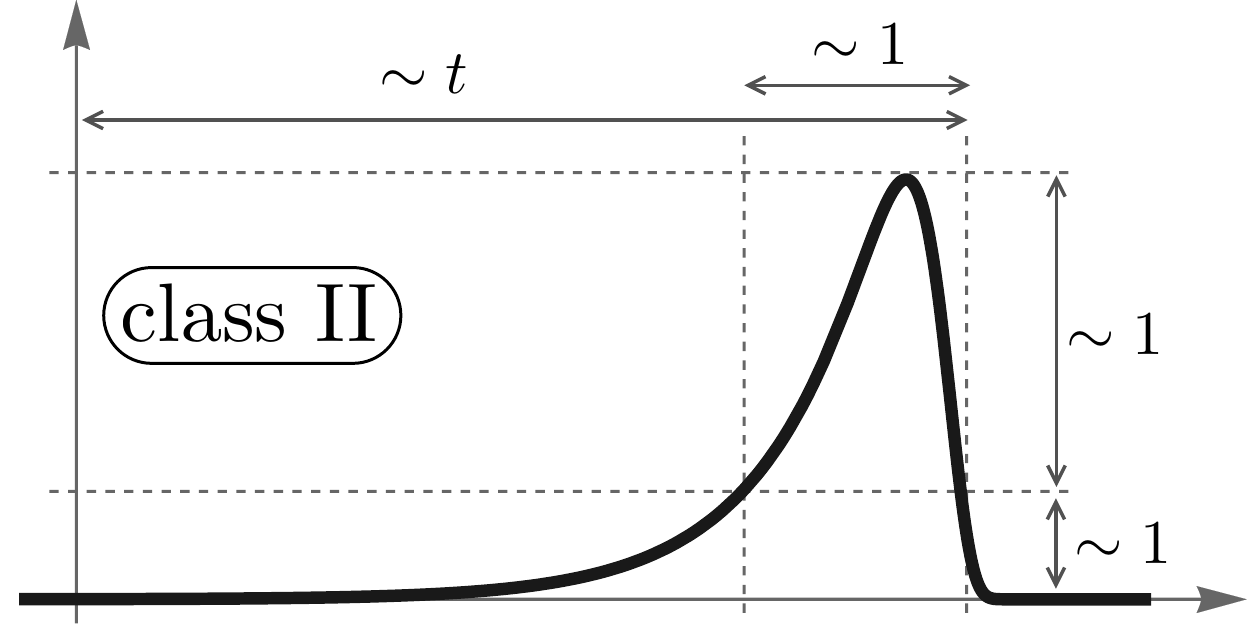}
  \hskip1cm
  \includegraphics[width=.4\textwidth]{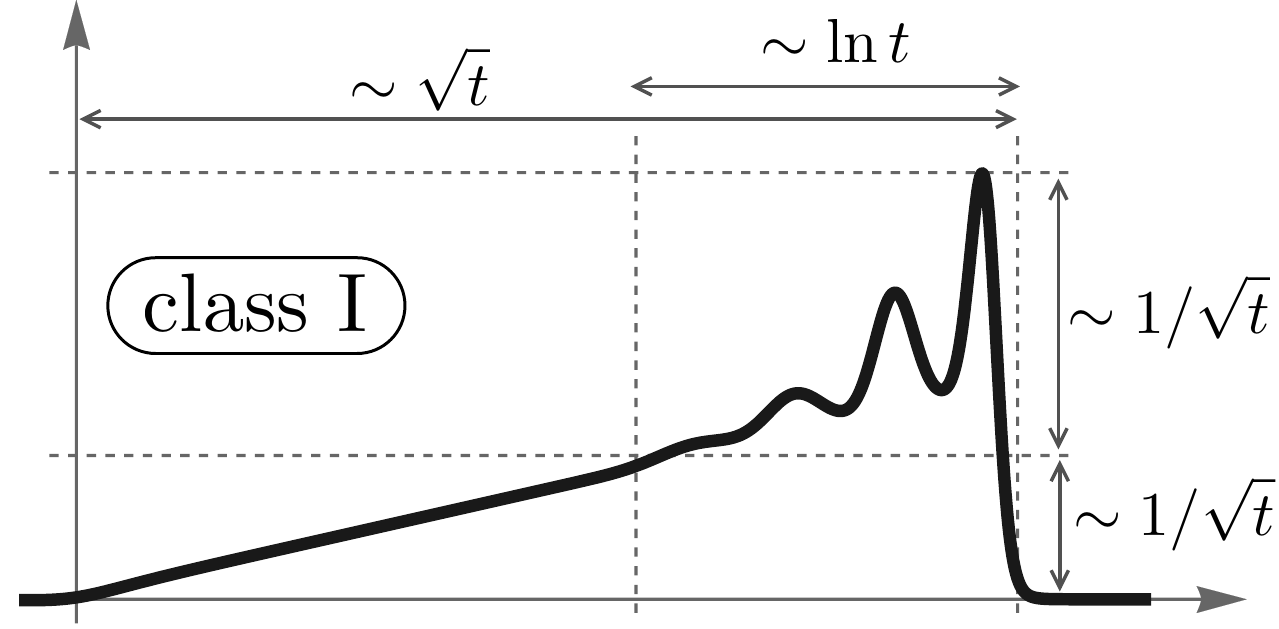}
\caption{Cartoon of large-time behaviour of solutions for the different type of kernels.}
  \label{fig:cartoon}
\end{figure}
 
}
%
\appendix
\section{Scheme for the numerical simulations}
%
%
The starting point for numerical simulations is the time-dependent problem \eqref{eq1} in an exponentially rescaled space variable
but as in Section \ref{Ss.diagonal} it is convenient to replace the scaling law \eqref{newvariables} by
\begin{align}
\label{app:Scaling}
\xi =2^X\,,\qquad X=\frac{\ln\xi}{\ln2}\,,\qquad T=t\ln2\,,\qquad \xi^2f(t,\xi)=u(t\ln{2},\ln\xi/\ln{2})\,.
\end{align}
 Moreover, it is also reasonable to normalize the kernel by
$\int_0^1K(x,1-x)\mathrm{d}x=1$, which implies $c_\alpha=\Gamma(2+2\alpha)/\Gamma(1+\alpha)^2$ for the family \eqref{kernelfamily}. 
\bigskip\par\noindent %
In the diagonal case, the nonlinear lattice equation \eqref{diagonal1} has a natural interpretation as a hierarchy of time-dependent ODEs, provided that the initial data are constant for $j< j_0$. In fact, $u_j(0)=c$ for all $j<j_0$ implies $u_j(t)=c$ for all $j<j_0$ and $t\geq0$, and by iteration we can hence regard \eqref{diagonal1} as an non-autonomous ODE for the output $u_j$ with 
known input $u_{j-1}$. This allows us to employ standard scalar ODE integrators, as for instance {\tt{DSolve}} in  {\sc Mathematica}, for the numerical solution of the discrete Burgers lattice 
\eqref{diagonal1}; cf. Figure \ref{fig:Nwavediagonal}.
\medskip
For non-diagonal kernels such as \eqref{kernelfamily}, the dynamical equation \eqref{eq1} can, thanks to  \eqref{uequation} and \eqref{app:Scaling}, be written as
\begin{align}
\label{app:Dynamics1}
\partial_T u(T,\cdot)=\mathcal{I}_{\mathrm{gain}}\big(u(T,\cdot)\big)-
\mathcal{I}_{\mathrm{loss}}\big(u(T,\cdot)\big)
\end{align}
with 
\begin{align}
\label{app:Integrals1}
\begin{split}
\mathcal{I}_{\mathrm{gain}}\big(u(T,\cdot)\big)|_X&:=\int_0^\infty W_{\mathrm{gain}}(Y)u\big(T,X-1-Y\big)u\big(T,X-1+\hat{Y}(Y)\big)\mathrm{d}{Y}\\
\mathcal{I}_{\mathrm{loss}}\big(u(T,\cdot)\big)|_X
&:=u(T,X)\int_{-\infty}^\infty
W_{\mathrm{loss}}(Y)u(T,X-Y)\mathrm{d}Y\,.
\end{split}
\end{align}
Here, the  weight functions are defined by
\begin{align}
\label{app:Weights}
W_{\mathrm{gain}}(Y):= \frac{K_\alpha\big(2^{Y+1}-1,1\big)}{\big(1-2^{-1-Y}\big)^2}\,,\qquad W_{\mathrm{loss}}(Y):= K_\alpha\big(2^{-Y},1\big)2^Y
\end{align}
and the function $\hat{Y}$ with
\begin{align}
\label{app:IntegralShift}
\hat{Y}(Y):=\frac{\ln(2-2^{-Y})}{\ln 2}
\end{align}
is strictly increasing for $Y>0$ and satisfies  $Y(0)=0$, $Y(+\infty)=1$, see Figure \ref{Fig.IntWeights1}.
\begin{figure}[ht!]
\centering{ %
\includegraphics[width=0.95\textwidth]{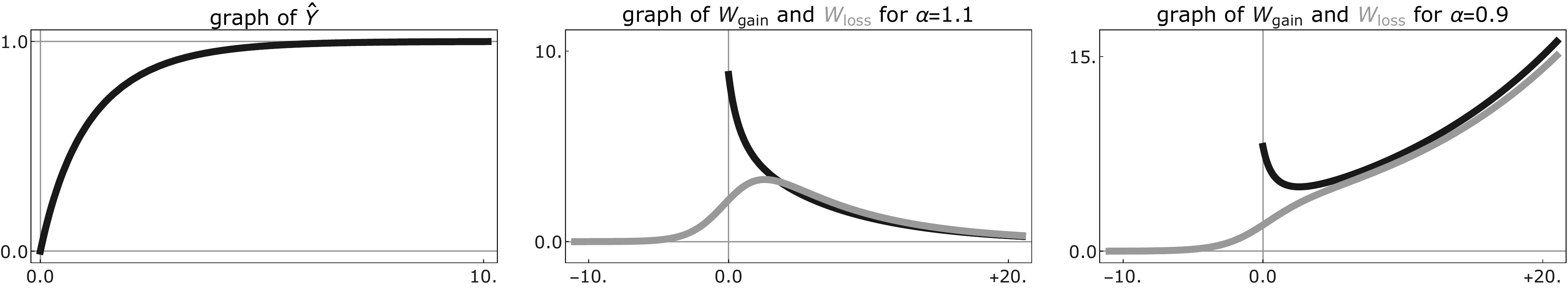}
} %
\caption{ Graph of the function $\widehat{Y}$ in \eqref{app:IntegralShift} as well as typical examples of the moment 
weights from \eqref{app:Weights} for the $\alpha$-kernel \eqref{kernelfamily} with $\alpha>1$ and $0<\alpha<1$.%
} %
\label{Fig.IntWeights1}
\end{figure}
As illustrated in Figure \ref{Fig.IntWeights1}, the weight functions from \eqref{app:Weights} can behave rather differently as $Y\to\infty$ and in the
case of non-decaying weight functions it is not advisable
to discretize the two integrals in \eqref{app:Integrals1} independently of each other. On the contrary, for the kernels in \eqref{kernelfamily} it is more convenient to 
reformulate \eqref{app:Dynamics1} as 
\begin{align}
\label{app:Dynamics2}
\partial_T u(T,\cdot)=\mathcal{I}_{\mathrm{A}}\big(U(T,\cdot)\big)+
\mathcal{I}_{\mathrm{B}}\big(u(T,\cdot)\big)
-\mathcal{I}_{\mathrm{C}}\big(u(T,\cdot)\big)\,,
\end{align}
where the three integral operators 
\begin{align}
\label{app:Integrals2}
\begin{split}
\mathcal{I}_\mathrm{A}\big(u(T,\cdot)\big)|_X&:= %
\int_0^\infty W_{\mathrm{gain}}(Y)u(T,X-1-Y)\Big(u\big(T,X-1+\hat{Y}(Y)\big)-u(T,X)\Big)\mathrm{d}{Y}\\
\mathcal{I}_\mathrm{B}\big(u(T,\cdot)\big)|_X&:=
u(T,X)\int_0^\infty \big(W_{\mathrm{gain}}(Y)-W_{\mathrm{loss}}(Y+1)\big)u(T,X-1-Y)\mathrm{d}{Y}
\\ %
\mathcal{I}_\mathrm{C}\big(u(T,\cdot)\big)|_X&:= %
u(T,X)\int_0^{\infty} W_{\mathrm{loss}}(1-Y)u(T,X-1+Y)\mathrm{d}{Y}\,,
\end{split}
\end{align}
possess better properties than $\mathcal{I}_{\mathrm{gain}}$ and $\mathcal{I}_{\mathrm{loss}}$. In fact, the weight functions for $\mathcal{I}_{\mathrm{B}}$  and $\mathcal{I}_{\mathrm{C}}$ decay exponentially as $Y\to\infty$, while the properties of $\hat{Y}$ ensure 
that  $u(T,X-1+\hat{Y}(Y))-u(T,X)$ decays faster than $1/W_{\mathrm{gain}}(Y)$ provided that $u$ is sufficiently regular at $(T,X)$.
All numerical data presented in this paper, see Figures \ref{fig:twaves}, \ref{fig:Nwaves},  are computed by a {\sc{Matlab}} implementation of the following, straight forward discretization of \eqref{app:Dynamics2}: 
\begin{enumerate}
\item 
Compute $u$ on the spatial grid $X\in\eps\mathbb{Z}\cap [0,L] $, where $1\ll L<\infty$ is a given discretization length and $0<\eps\ll1$ a chosen spacing.
\item 
Continue $u$ constantly for $X<0$ and $X>L$ by constants $c_{-}$ and $c_+$, respectively, with $c_-=c_+=0$ for integrable initial data and 
$c_->0$, $c_+=0$ in order to compute traveling waves.
\item 
Approximate all integrals from \eqref{app:Integrals2} by Riemann sums with respect to $Y\in\eps\mathbb{Z}\cap[0,R]$, where $R>1$ is another discretization parameter.
\item 
Use the explicit Euler scheme with step size $0<\tau\ll1$ for the time integration.
\end{enumerate}
The resulting numerical scheme, however, is neither very accurate nor fast. It remains a challenging task to construct alternative algorithms that allow
to resolve the long time behavior of coagulation equations with diagonal dominant kernels of homogeneity $1$ more efficiently.
In previous work Filbet and Lauren\c{c}ot \cite{FL04} developed a finite volume scheme to simulate the coagulation equation in the conservative form \eqref{eq2}, but did not study
the long-time behaviour specifically for kernels with homogeneity one.
Lee \cite{Lee01} simulated the discrete version of the coagulation equation. He considered the case of class-I kernels with homogeneity one, but the convergence
of the algorithm is slow in this case and the results are not completely conclusive.
It seems that no numerical simulations of the coagulation equation have been previously performed for the equation in exponential variables \eqref{uequation}. 
\medskip
We also mention that the dynamical solutions computed with the traveling wave continuation $u(T,X)=c_->0$ and $u(T,X)=c_+=0$ for $X<0$ and $X>L$, respectively,
might be unphysical for small times due to artificial  boundary effects. For sufficiently large times and sufficiently large computational domains, however, those 
numerical solutions approach a traveling wave that connects $c_-$ to $c_+$.


 \paragraph{Acknowledgment.}
  The authors acknowledge support through the CRC 1060 \textit{The mathematics of emergent effects} at the University of Bonn that is funded through the German Science Foundation (DFG).

{\small
 \bibliographystyle{plain}%

}

\end{document}